\documentclass[reqno, 12pt, reqno]{amsart}

\usepackage{fullpage}
\usepackage[english]{babel}
\usepackage[T1]{fontenc}
\usepackage{graphicx}
\usepackage{mathtools}

\usepackage{etex}
\usepackage[utf8x]{inputenc}
\usepackage{amsmath,amssymb,amsfonts,amsthm}
\usepackage{geometry}
\theoremstyle{plain}
\newtheorem{thm}{Theorem}[section]
\newtheorem{cor}[thm]{Corollary}
\newtheorem{lem}[thm]{Lemma}
\newtheorem{prop}[thm]{Proposition}
\newtheorem*{proposition*}{Proposition $2.4$}
\newtheorem{rem}{Remark}

\theoremstyle{definition}
\newtheorem{defin}[thm]{Definition}
\newtheorem{esem}[thm]{Example}

\DeclareMathOperator{\Gal}{Gal}

\DeclareMathOperator{\bfu}{\mathbf{u}}

\DeclareMathOperator{\bfz}{\mathbf{z}}
\DeclareMathOperator{\bfv}{\mathbf{v}}

\DeclareMathOperator{\rank}{rank}
\DeclareMathOperator{\sing}{sing}
\DeclareMathOperator{\Swan}{Swan}
\DeclareMathOperator{\disc}{disc}

\DeclareMathOperator{\Emb}{Emb}

\DeclareMathOperator{\Id}{Id}
\DeclareMathOperator{\Fr}{Fr}

\DeclareMathOperator{\Tr}{Tr}

\theoremstyle{remark}

\begin{document}
\title{On the 2-torsion in class groups of number fields}
\author{Dante Bonolis}
\address{TU Graz, Steyrergasse 30, Graz, Austria}
\email{dante.bonolis@tugraz.at}

\begin{abstract}
  In $2020$, Bhargava, Shankar, Taniguchi, Thorne, Tsimerman, and Zhao proved that for a finite extension $K/\mathbb{Q}$ of degree $n\geq 5$, the size of the $2$-torsion class group is bounded by $ h_{2}(K)=O_{n,\varepsilon}(D_{K}^{\frac{1}{2}-\frac{1}{2n}+\varepsilon})$, where $D_{K}$ is the absolute discriminant of $K$. In the present paper, we improve their bound by proving that $ h_{2}(K)=O_{n,\varepsilon}(D_{K}^{\frac{1}{2}-\frac{1}{2n}-\delta_{K}+\varepsilon})$, for a constant $\delta_{K}\geq\frac{1}{28n}-\frac{3}{28n(n-1)}$.
\end{abstract}
\date{\today}
\maketitle
\section{Introduction}
For a given number field $K/\mathbb{Q}$ of degree $n\geq 2$ and any prime $\ell$, a classical problem in algebraic number theory is to bound $h_\ell(K)$, the size of the $\ell$-torsion subgroup $\text{Cl}(K)[\ell]$ of the class group $\text{Cl}(K)$. Since $h_{\ell}(K)\leq h(K)=\#\text{Cl}(K)$, Landau's classical bound (see, for example, \cite[(1.1)]{LZ25}) gives
\[
h_{\ell}(K)=O_{n}\bigl(D_K^{1/2}(\log D_K)^{n-1}\bigr),
\]
where $D_K$ denotes the absolute value of the discriminant of $K$ over $\mathbb{Q}$. Recently, Lemke Oliver and Zaman
(\cite[Corollary 1]{LZ25}) showed that, for every fixed
$n$ and $\ell$, one has $h_{\ell}(K)=o_{n,\ell}(D_K^{1/2})$ as $D_K\to\infty$, with an explicit saving by a power of $\log D_K$. On the other hand, it is conjectured that
\begin{equation}
h_{\ell}(K) = O_{n,\ell,\varepsilon}(D_{K}^{\varepsilon}),
\label{conj : ell}
\end{equation}
for any $\varepsilon > 0$. This conjecture remains wide open. For prime $\ell$ and $n\geq 2$, the only case known for all number fields of degree $n$ is $(n,\ell)=(2,2)$, which follows from Gauss' genus theory. In fact, even establishing the existence of a constant $\delta_{n,\ell} > 0$ such that for any number field $K$ of degree $n$, we have
\[
h_{\ell}(K) =O_{n,\ell,\varepsilon} (D_{K}^{\frac{1}{2}-\delta_{n,\ell}+\varepsilon})
\]
is already a challenging problem. In this paper, we focus on the case $\ell=2$. It was shown in \cite[Theorem 1.1]{BSTTZ20} that one may take
\[
\delta_{n,2} =
\begin{cases}
    0.2215\ldots & \text{if } n = 3,4, \\
    \frac{1}{2n} & \text{if } n \geq 5.
\end{cases}
\]
Our main result is the following:

\begin{thm}
Let $K/\mathbb{Q}$ be a number field of degree $n\geq 5$. Then there exists $\delta_{K}>0$ such that for any $\varepsilon >0$, one has
\[
h_{2}(K)=O_{n,\varepsilon}(D_{K}^{\frac{1}{2}-\frac{1}{2n}-\delta_{K}+\varepsilon}).
\]
Moreover, $\delta_{K}\geq \frac{1}{28n}-\frac{3}{28n(n-1)}$.
\label{thm : main}
\end{thm}

\section{Outline of the proof}
We start by fixing some notation: let $\mathcal{O}_{K}$ be the ring of integers of $K$. It is a well-known fact that under the Minkowski embedding, $\mathcal{O}_K$ is a lattice of rank $n$ and covolume $\asymp_{n} D_K^{1/2}$. For every $\alpha \in \mathcal{O}_{K}$, consider the norm:
\[
\|\alpha\|=\sum_{\sigma\in \Emb(K/\mathbb{Q})}|\sigma(\alpha)|,
\]
where $\Emb(K/\mathbb{Q})$ denotes the set of embeddings of $K$ into $\mathbb{C}$ fixing $\mathbb{Q}$. Since $\mathcal{O}_{K}$ is a lattice, we can find a basis $\{\alpha_{1,K}=1, \alpha_{2,K}, \ldots, \alpha_{n,K}\}$ which is Minkowski-reduced, i.e. such that $1 \leq \|\alpha_{2,K}\|\leq\cdots\leq \|\alpha_{n,K}\|$, and for each $2\leq i\leq n$, $\|\alpha_{i,K}\|\ll_{n}\|v\|$ for every $v\in\mathcal{O}_K\setminus\operatorname{Span}_{\mathbb{Q}}\{1,\alpha_{2,K},\ldots,\alpha_{i-1,K}\}$. Observe that this implies that
\begin{equation}
\|\alpha_{2,K}\|\cdots\|\alpha_{n,K}\|\asymp D_{K}^{\frac{1}{2}}.
\label{eq : succmindisc}
\end{equation}
Finally, for each $i=2,...,n$, we denote by $\delta_{i,K}$ the positive real number such that $\|\alpha_{i,K}\|= D_{K}^{\delta_{i,K}}$. At this point, it is useful to record the following
\begin{prop}
    Let $K/\mathbb{Q}$ be a finite extension of degree $n$. Then
    \begin{itemize}
        \item[$i)$] $\|\alpha_{n,K}\|=O_{n}(D_{K}^{\frac{1}{n}})$. In particular we have $0\leq \delta_{i,K}\leq \frac{1}{n}+O_{n}\left(\frac{1}{\log D_{K}}\right)$, for each $2\leq i\leq n$;
        \item[$ii)$] $\|\alpha_{2,K}\|=O_{n}(D_{K}^{\frac{1}{2(n-1)}})$. In particular we have $0\leq \delta_{2,K}\leq \frac{1}{2(n-1)}+O_{n}\left(\frac{1}{\log D_{K}}\right)$.
    \end{itemize}
\end{prop}
\begin{proof}
    We start with $(i)$. The bound $\|\alpha_{n,K}\|=O_{n}(D_{K}^{\frac{1}{n}})$ is proved in \cite[Theorem $1.6$]{BSTTZ20}. Since the basis  $\{1, \alpha_{2,K}, \ldots, \alpha_{n,K}\}$ is Minkowski-reduced, we have  $\|\alpha_{i,K}\|\leq \|\alpha_{n,K}\|=O(D_{K}^{\frac{1}{n}})$ for every $i\in\{2,...,n\}$. Hence, $\delta_{i,K}\leq\frac{1}{n}+O_{n}\left(\frac{1}{\log D_{K}}\right)$ for every $i\in\{2,...,n\}$, as claimed.\newline
    For $(ii)$, it is enough to observe that by $(\ref{eq : succmindisc})$,
    \[
    D_{K}^{\frac{1}{2}}\gg_{n} \|\alpha_{2,K}\|\cdots\|\alpha_{n,K}\|\geq \|\alpha_{2,K}\|^{n-1}.
    \]
\end{proof}

Our starting point is the same as \cite[Theorem $1.1$]{BSTTZ20}: 
The size of the $2$-torsion of a number field extension is bounded by
\[
h_{2}(K)\ll\# \left\{(y,z_{1},...,z_{n})\in\mathbb{Z}^{n+1}:\begin{matrix}|z_{i}|\ll \frac{D_{K}^{\frac{1}{n}}}{\|\alpha_{i,K}\|},\quad\text{for }i=1,...,n\\y^{2}=N_{K}(z_{1}+z_{2}\alpha_{2,K}+\cdots+z_{n}\alpha_{n,K})\end{matrix}     \right\}.
\]
For notational simplicity, all boxes below are understood to have their side lengths multiplied by a sufficiently large fixed constant depending only on $n$. Our estimates are uniform under such fixed dilatation.
In \cite{BSTTZ20}, the bound for $h_{2}(K)$ (when $n\geq 5$) is achieved by fixing a choice of $z_{2},...,z_{n}$ and by using the Bombieri-Pila bound (\cite{BP89}) to count integral solutions $(y,x_{1})$ on the curve, 
\[
Y^{2}=N_{K}(X_{1}+z_{2}\alpha_{2,K}+\cdots+z_{n}\alpha_{n,K}),
\]
in the box $[-D_{K}^{\frac{1}{2}},D_{K}^{\frac{1}{2}}]\times [-D_{K}^{\frac{1}{n}},D_{K}^{\frac{1}{n}}]$. Similarly to Salberger's approach (\cite{Sal22}), we shall instead fix only $ z_{3},...,z_{n}$ and count integral solutions $(y,x_{1},x_{2})$ on the surface
\[
Y^{2}=N_{K}(X_{1}+X_{2}\alpha_{2,K}+z_{3}\alpha_{3,K}+\cdots+z_{n}\alpha_{n,K}),
\]
in the box 
\begin{equation}
[-D_{K}^{\frac{1}{2}},D_{K}^{\frac{1}{2}}]\times [-D_{K}^{\frac{1}{n}},D_{K}^{\frac{1}{n}}]\times [-D_{K}^{\frac{1}{n}-\delta_{2,K}},D_{K}^{\frac{1}{n}-\delta_{2,K}}].
\label{eq : BoxB}
\end{equation}
Salberger treats these surfaces using the determinant method, whereas we use the square sieve and, in the shorter $X_{2}$-range, the $q$-analogue of van der Corput's method.
One of the main difficulties in using this approach is that the box in $(\ref{eq : BoxB})$ has edges with highly unequal sizes: for instance, when $\frac{1}{2n} \leq \delta_{2,K} \leq \frac{1}{2(n-1)}+O_{n}\left(\frac{1}{\log D_{K}}\right)$, the range of the variable $X_{2}$ is at or below the square-root threshold of the range of the variable $X_{1}$, making the analysis significantly more challenging. As a result, we are going to use a different approach depending on the size of $\delta_{2,K}$.
\subsection{ The case $ \delta_{2,K} \leq \frac{1}{2n}$.}
In this case, it is enough to deploy the classical square sieve, as introduced in \cite{HB84}, to obtain the following theorem.
\begin{thm}
Let $n \geq 5$ and $\varepsilon > 0$. For any number field $K$ of degree $n$, we
have
\[
h_{2}(K)\ll_{n,\varepsilon} D_{K}^{\frac{1}{2}-\frac{1}{2n}-\left(\frac{1}{6n}-\frac{\delta_{2,K}}{3}\right)+\varepsilon}.
\]
\label{thm : deltalarge}
\end{thm}
This gives a positive saving over \cite[Theorem 1.1]{BSTTZ20} when $\delta_{2,K}<\frac{1}{2n}$. At $\delta_{2,K}=\frac{1}{2n}$ the two bounds coincide, whereas for $\delta_{2,K}>\frac{1}{2n}$ the bound in Theorem \ref{thm : deltalarge} is weaker. Hence, Theorem \ref{thm : deltalarge} alone does not suffice to prove Theorem \ref{thm : main}.
\subsection{ The case $ \frac{1}{2n}\leq \delta_{2,K} \leq \frac{1}{2(n-1)}+O_{n}\left(\frac{1}{\log D_{K}}\right)$.}
As mentioned earlier, this case will be the most complicated due to the fact that the range of the variable $X_{2}$ is considerably smaller than the one of $X_{1}$. To address this, we combine the square sieve with the $q$-analogue of van der Corput's method, introduced by Heath-Brown; see \cite{HB78}, \cite{Pie06}, and \cite[Section 3]{WX21}. The differencing step exploits the factorization of the sieve modulus and leads to the complete correlation sums in Proposition \ref{prop : Anumber}. We obtain the following theorem.
\begin{thm}
Let $n\geq 5$ be an integer and $\varepsilon > 0$. For any number field $K$ of degree $n$, we
have
\[
h_{2}(K)\ll_{n,\varepsilon} D_{K}^{\frac{1}{2}-\frac{1}{2n}-\left(\frac{1}{7n}-\frac{3\delta_{2,K}}{14}\right)+\varepsilon},
\]
provided that $\frac{8}{21n}\leq\delta_{2,K}$.
\label{thm : deltasmall}
\end{thm}
Set $c_{n}=\frac{1}{28n}-\frac{3}{28n(n-1)}$. If $\delta_{2,K}<\frac{8}{21n}$, then
\[
\frac{1}{6n}-\frac{\delta_{2,K}}{3}\geq \frac{5}{126n}>c_{n},
\]
so Theorem \ref{thm : deltalarge} gives the desired bound. If $\delta_{2,K}\geq\frac{8}{21n}$, then Theorem \ref{thm : deltasmall} and Proposition $2.1(ii)$ give
\[
\begin{split}
h_{2}(K)&\ll_{n,\varepsilon}D_K^{\frac12-\frac{1}{2n}-\frac{1}{7n}+\varepsilon}\|\alpha_{2,K}\|^{3/14}\\&\ll_{n,\varepsilon}D_K^{\frac12-\frac{1}{2n}-\frac{1}{7n}+\frac{3}{28(n-1)}+\varepsilon}\\&=D_K^{\frac12-\frac{1}{2n}-c_{n}+\varepsilon}.
\end{split}
\]
This proves Theorem \ref{thm : main}. The main ingredient in proving Theorem \ref{thm : deltasmall} is the following stratification result for exponential sums.
\begin{prop}
Let $p$ be an odd prime and let $F(X,Y)\in\mathbb{F}_{p}[X,Y]$ be a square-free polynomial such that
\[
F(X,Y)=\prod_{i=1}^{n}(X+a_{i}Y+b_{i})\text{ over }\overline{\mathbb{F}}_{p},
\]
and assume that the following conditions hold:
\begin{itemize}
\item[$i)$] There exists $i\in\{2,...,n\}$ such that $a_{1}\neq a_{i}$.
\item[$ii)$] $(a_{i}-a_{1})(b_{j}-b_{1}) \neq(a_{j}-a_{1})(b_{i}-b_{1})$ for any $i,j\in\{2,..,n\}$ such that $a_{j}\neq a_{1}$ and $i\neq j$.
\item[$iii)$] For every $i,j$ with $i\neq j$, $b_{i}\neq b_{j}$.
\item[$iv)$] For every $y\in\overline{\mathbb{F}}_{p}$, the polynomial $F(X,y)$ is not a square over $\overline{\mathbb{F}}_{p}[X]$.
\end{itemize}
Then, for every $s\in\mathbb{F}_{p}$, there exists $A_s\subset\mathbb{F}_{p}^2$ such that
\begin{itemize}
\item[$i)$] for any $(r,t)\in (\mathbb{F}_{p}\times\mathbb{F}_{p})\setminus A_{s}$, one has 
\[
\left|\sum_{(x,y,z)\in\mathbb{F}_{p}^{3}}\left(\frac{F(x,y)}{p}\right)\left(\frac{F(z,y+r)}{p}\right)e_{p}(s(x-z)+ty)\right|\ll_{n} p^{3/2}(s,t,p)^{1/2}.
\]
\item[$ii)$] $\# A_{s}=O_{n}(1)$.
\end{itemize}
\label{prop : Anumber}
\end{prop}
When $(s,p)=1$, Proposition \ref{prop : Anumber} implies that the exponential sum
\begin{equation}
\sum_{(x,y,z)\in\mathbb{F}_{p}^{3}}\left(\frac{F(x,y)}{p}\right)\left(\frac{F(z,y+r)}{p}\right)e_{p}(s(x-z)+ty)
\label{eq : Anumber}
\end{equation}
exhibits square root cancellation for all $(r,t)$ which are not contained in a set of dimension $0$ in $\mathbb{F}_{p}^{2}$.  Then Proposition $\ref{prop : Anumber}$ is the analogue of showing that the $A$-number associated to the exponential sum in $(\ref{eq : Anumber})$, in the sense of \cite{Kat89}, is not zero. We remark that, in general, determining whether the $A$-number associated with an exponential sum is not zero is a difficult problem (see \cite{Kat94}, \cite{FK01}).
\section{Previous results}
The quantity $h_{\ell}(K)$, for a number field $K$ of degree $n$, has been the subject of extensive study. As mentioned in the introduction, Gauss' genus theory establishes (\ref{conj : ell}) for quadratic fields when $\ell=2$. Further cases are known for restricted families of number fields, as indicated below.
The following table summarizes some known pointwise bounds of the form
\[
h_{\ell}(K)\ll_{n,\ell,\varepsilon}D_K^{\frac{1}{2}-\delta+\varepsilon}.
\]
Here $L$ denotes the Galois closure of $K/\mathbb{Q}$.\newpage

\begin{table}[h]
\centering
\resizebox{\textwidth}{!}{%
\begin{tabular}{|l|c|c|c|l|}
\hline
\textbf{Source} & $\boldsymbol{\delta}_{n,\ell}$ & $\mathbf{n}$ & $\boldsymbol{\ell}$ & \textbf{Other restrictions} \\
\hline
Gauss & 1/2 & 2 & 2 & \\
\hline
\cite{Pie05,Pie06} &  $1/56$ & 2 & 3 & \\
\hline
\cite{HV06} & $0.0582\ldots$ & 2 & 3 & \\
\hline
\cite{EV07} & $ 1/6$ & 2, 3 & 3 & \\
\hline
\cite{CK25} & $ 0.1806$ & 2 & 3 & \\
\hline
\cite{HB24} & $ 1/(2\ell)$ & 2, 3 & prime & $p \mid D_K \Rightarrow p \leq D_K^{\delta_\ell}$ \\
\hline
\cite{HB24} & $ 1/(4\ell)$ & 3 & prime & pure cubic extension \\
\hline
\cite{BSTTZ20} & $ 0.2215\ldots$ & 3, 4 & 2 & \\
\hline
\cite{EV07} & $ 1/168$ & 4 & 3 & $\mathrm{Gal}(L/\mathbb{Q}) \cong A_4$ or $S_4$ \\
\hline
\cite{EV07} & $ \delta$ & 4 & 3 & $\mathrm{Gal}(L/\mathbb{Q}) \cong C_4$, $V_4$, or $D_4$ \\
\hline
\cite{BSTTZ20} & $ 1/(2n)$ & $\geq 5$ & 2 & \\
\hline
\cite{Sal22} &
$\frac{2-n\delta_{2,K}-\sqrt{2-2n\delta_{2,K}}}{n}$ & $\geq 5$ & 2 & $\delta_{2,K}\leq 1/(2n)$\\
\hline
\cite{KW22} & 1/2 & $p^r$ & $p^s$ & $\mathrm{Gal}(K/\mathbb{Q}) \cong G$, any $p$-group $G$, $K=L$ \\
\hline
\cite{Wan21,Wan20} & $ \delta_{G,\ell}$ & powerful $>2$ & $\ell\geq 2$ & $\mathrm{Gal}(K/\mathbb{Q}) \cong G$, many nilp. grps $G$, $K=L$ \\
\hline
\end{tabular}
}
\end{table}
We remark that Theorem \ref{thm : deltalarge} improves upon \cite{Sal22} for all number fields $K$ with $0\leq\delta_{2,K}<1/(2n)$, and gives the same exponent when $\delta_{2,K}=1/(2n)$. Significant efforts have been made to achieve unconditional improvements on the trivial bound \textit{on average across families of number fields}: \cite{HP17}, \cite{EPW17}, \cite{Wid18}, \cite{FW18}, \cite{PTBW20}, \cite{An20}, \cite{FW21}, \cite{TZ23}, \cite{LOTZ24}, \cite{KT24}, \cite{LOWW25}, and \cite{CK25}. 
\subsection{Some applications of Theorem \ref{thm : deltalarge}}
We present two consequences of Theorem \ref{thm : deltalarge}.
\begin{cor}
    Let $\mathbb{Q}\subsetneq F$ be a finite extension of $\mathbb{Q}$. Then for every $\mathbb{Q}\subsetneq F\subset K$, with $[K:\mathbb{Q}]=n\geq 5$, we have
    \[
    h_{2}(K)\ll_{D_{F},n,\varepsilon}D_{K}^{\frac{1}{2}-\frac{2}{3n}+\varepsilon}.
    \]
\end{cor}
\begin{proof}
Since $\alpha_{2,F}\in\mathcal{O}_{K}\setminus\mathbb{Q}$,
the reduced-basis property gives
\[
\|\alpha_{2,K}\|_{K}\ll_{n} \|\alpha_{2,F}\|_{K}=[K:F]\|\alpha_{2,F}\|_{F}\ll_{n,D_F}1,
\]
where the subscripts indicate the field in which the norm is computed. Hence Theorem \ref{thm : deltalarge} gives
\[
h_{2}(K)\ll_{n,\varepsilon}D_{K}^{\frac12-\frac{2}{3n}+\varepsilon}\|\alpha_{2,K}\|_{K}^{1/3}\ll_{n,D_F,\varepsilon}D_{K}^{\frac12-\frac{2}{3n}+\varepsilon}.
\]
\end{proof}

\begin{cor}
Let $\varepsilon > 0$ and let $q$ be a prime number. Then for every $n\geq 5$,
\[
h_{2}(\mathbb{Q}[\sqrt[n]{q}])\ll_{n,\varepsilon} D_{\mathbb{Q}[\sqrt[n]{q}]}^{\frac{1}{2}-\frac{2}{3n}+\frac{1}{3n(n-1)}+\varepsilon}.
\]
\end{cor}
\begin{proof}
Put $K=\mathbb{Q}(\sqrt[n]{q})$ and $\theta=\sqrt[n]{q}$. Since $X^n-q$ is Eisenstein at $q$, we have $[K:\mathbb{Q}]=n$, and $q$ is totally ramified in $K$. Consequently, $q^{n-1}\mid D_K$. Since $\theta\in\mathcal{O}_K\setminus\mathbb{Q}$, we obtain
\[
\|\alpha_{2,K}\|\ll_{n} \|\theta\|=nq^{1/n}\ll_{n} D_K^{\frac{1}{n(n-1)}}.
\]
Applying Theorem \ref{thm : deltalarge}, we conclude that
\[
h_{2}(K)\ll_{n,\varepsilon}D_K^{\frac12-\frac{2}{3n}+\varepsilon}\|\alpha_{2,K}\|^{1/3}\ll_{n,\varepsilon}D_K^{\frac12-\frac{2}{3n}+\frac{1}{3n(n-1)}+\varepsilon}.
\]
\end{proof}
\section{Proof of Theorem $\ref{thm : main}$}
\subsection{Some Lemmas}
As mentioned in the introduction, the idea is to deploy the square sieve to detect how often the norm
\[
N_{K}(X_{1}+X_{2}\alpha_{2,K}+z_{3}\alpha_{3,K}+\cdots+z_{n}\alpha_{n,K}) =: F_{\bfz}(X_{1},X_{2}),
\]
is a square, for $(x_{1},x_{2}) \in \left[-D_{K}^{\frac{1}{n}}, D_{K}^{\frac{1}{n}}\right] \times \left[-D_{K}^{\frac{1}{n}-\delta_{2,K}}, D_{K}^{\frac{1}{n}-\delta_{2,K}}\right]\cap\mathbb{Z}^{2}$, and where
\[
\bfz\in \left[-D_{K}^{\frac{1}{n}-\delta_{3,K}}, D_{K}^{\frac{1}{n}-\delta_{3,K}}\right] \times \cdots \times \left[-D_{K}^{\frac{1}{n}-\delta_{n,K}}, D_{K}^{\frac{1}{n}-\delta_{n,K}}\right]\cap\mathbb{Z}^{n-2}
\]
is fixed. Before applying the sieve, we must exclude certain choices of $\bfz$ for which our method does not apply. Namely:
\begin{itemize}
\item[$(i)$] Those $\bfz\in\mathbb{Z}^{n-2}$ for which $F_{\bfz}(X_{1},X_{2})$ is a square in $\mathbb{Z}[X_{1},X_{2}]$. We denote this set by $Z_{\square}$.
    \item[$(ii)$] Those $\bfz\in\mathbb{Z}^{n-2}$ for which the polynomial $F_{\bfz}(X_{1},X_{2})$ does not satisfy the hypotheses of Proposition \ref{prop : Anumber} over $\overline{\mathbb{Q}}$. We denote the set of these points as $Z_{F}$.
\end{itemize}
To simplify the notation, from now on we denote
\[
\mathcal{B}= \left[-D_{K}^{\frac{1}{n}}, D_{K}^{\frac{1}{n}}\right] \times \left[-D_{K}^{\frac{1}{n}-\delta_{2,K}}, D_{K}^{\frac{1}{n}-\delta_{2,K}}\right],
\]
and
\[
\mathcal{C} = \left[-D_{K}^{\frac{1}{n}-\delta_{3,K}}, D_{K}^{\frac{1}{n}-\delta_{3,K}}\right] \times \cdots \times \left[-D_{K}^{\frac{1}{n}-\delta_{n,K}}, D_{K}^{\frac{1}{n}-\delta_{n,K}}\right].
\]
 \subsubsection{The contribution of $Z_{\square}$} 
We start by recording the following Lemma, which is proved in \cite[Bottom of page $6$]{BSTTZ20} and will be useful in bounding the contribution of the $\bfz\in Z_{\square}$:
\begin{lem}
    Let $K/\mathbb{Q}$ be a number field and assume that there exists $\mathbb{Q}\subset F\subset K$, such that $[K:F]=2$. Then $ h_{2}(K)\ll_{n,\varepsilon} D_{K}^{\frac{1}{4}+\varepsilon}$.
    \label{lem : index2}
\end{lem}
\begin{cor}
    Assume that $K/\mathbb{Q}$ does not contain a subfield of index $2$. Then $Z_{\square}$ is empty if $n=5,6,7$, and 
    \[
    \sum_{\substack{\bfz\in\mathcal{C}\cap Z_{\square}}}\sum_{\substack{(x_{1},x_{2})\in\mathcal{B}\cap\mathbb{Z}^{2}\\F_{\bfz}(x_{1},x_{2})=\square}}1\ll_{n,\varepsilon} D_{K}^{\frac{1}{4}+\frac{1}{n}+\varepsilon},
    \]
    for $n\geq 8$.
     \label{cor : square}
\end{cor}
\begin{proof}
Assume that the polynomial $F_{\bfz}(X_{1},X_{2})$ is a square. Consider the field $K'=\mathbb{Q}[\alpha_{2,K},z_{3}\alpha_{3,K}+\cdots+z_{n}\alpha_{n,K}]\subset K$. Since we are assuming 
\[F_{\bfz}(X_{1},X_{2})=N_{K/\mathbb{Q}}(X_{1}+X_{2}\alpha_{2,K}+z_{3}\alpha_{3,K}+\cdots+z_{n}\alpha_{n,K})\quad\text{ is a square,}
\] 
then for all $\beta\in K'$, $N_{K/\mathbb{Q}}(\beta)$ is a square. Thus, $2|[K:K']$. Since by hypothesis, $K$ does not contain any subfield of index $2$, it follows that $[K:K']\geq 4$. Hence, the set $Z_{\square}$ is actually empty when $n=5,6,7$. Let us assume now $n\geq 8$. We start by writing
\[
\begin{split}
\sum_{\substack{\bfz\in\mathcal{C}\cap Z_{\square}}}\sum_{\substack{(x_{1},x_{2})\in\mathcal{B}\cap\mathbb{Z}^{2}\\F_{\bfz}(x_{1},x_{2})=\square}}1&=\sum_{\substack{\bfz\in\mathcal{C}\cap\mathbb{Z}^{n-2}\\ F_{\bfz}(X_{1},X_{2})=\square}}\sum_{\substack{(x_{1},x_{2})\in\mathcal{B}\cap\mathbb{Z}^{2}\\F_{\bfz}(x_{1},x_{2})=\square}}1\\&\leq \sum_{\substack{(x_{2},\bfz)\in\mathbb{Z}^{n-1}\\(x_{2},\bfz)\in[-D_{K}^{\frac{1}{n}-\delta_{2,K}},D_{K}^{\frac{1}{n}-\delta_{2,K}}]\times \mathcal{C}\\ F_{\bfz}(X_{1},x_{2})=\square}}\quad\sum_{\substack{|x_{1}|\leq D_{K}^{\frac{1}{n}}\\F_{\bfz}(x_{1},x_{2})=\square}}1.
\end{split}
\]
An application of \cite[Lemma 4.1]{BSTTZ20} leads to
\[
\begin{split}
 \sum_{\substack{(x_{2},\bfz)\in\mathbb{Z}^{n-1}\\(x_{2},\bfz)\in[-D_{K}^{\frac{1}{n}-\delta_{2,K}},D_{K}^{\frac{1}{n}-\delta_{2,K}}]\times \mathcal{C}\\ F_{\bfz}(X_{1},x_{2})=\square}}1 \ll_{n,\varepsilon} D_{K}^{\frac{1}{4}+\varepsilon}
\end{split}
\]
Hence
\[
\sum_{\substack{(x_{2},\bfz)\in\mathbb{Z}^{n-1}\\(x_{2},\bfz)\in[-D_{K}^{\frac{1}{n}-\delta_{2,K}},D_{K}^{\frac{1}{n}-\delta_{2,K}}]\times \mathcal{C}\\ F_{\bfz}(X_{1},x_{2})=\square}}\quad\sum_{\substack{|x_{1}|\leq D_{K}^{\frac{1}{n}}\\F_{\bfz}(x_{1},x_{2})=\square}}1\ll_{n,\varepsilon} D_{K}^{\frac{1}{4}+\frac{1}{n}+\varepsilon},
\]
as claimed.
\end{proof}
 \subsubsection{The contribution of $Z_{F}$}
  We now bound the contribution of the vectors in $\mathcal{C}\cap Z_F$. For $\bfz=(z_3,\ldots,z_{n})$, write $z=z_3\alpha_{3,K}+\cdots+z_{n}\alpha_{n,K}$. Since we have
 \[
 \begin{split}
 F_{\bfz}(X_{1},X_{2})&=\prod_{\sigma\in\Emb (K)}(X_{1}+\sigma(\alpha_{2,K})X_{2}+\sigma(z))\text{ over }\overline{\mathbb{Q}},
 \end{split}
 \]
the polynomial $F_{\bfz}(X_{1},X_{2})$ satisfies the hypotheses of Proposition $\ref{prop : Anumber}$ over $\overline{\mathbb{Q}}$, if
\begin{enumerate}
    \item[$a)$] the element $z$ does not belong to any
    proper subfield of $K$ (this also guarantees that $F_{\bfz}$ is square-free), i.e. $\mathbb{Q}(z)=K$;
    \item[$b)$] for every $\sigma, \tau\in\Emb(K)\setminus\{\Id\}$, with $\sigma\neq \tau$, and $\tau(\alpha_{2,K})\neq \alpha_{2,K}$, one has
    \[
    \sigma (z)-z\neq\frac{\sigma(\alpha_{2,K})-\alpha_{2,K}}{\tau(\alpha_{2,K})-\alpha_{2,K}}(\tau(z)-z).
    \]
    \item[$c)$] for every $x_{2}\in\overline{\mathbb{Q}}$, the polynomial $F_{\bfz}(X_{1},x_{2})$ is not a square in $\overline{\mathbb{Q}}[X_{1}]$.
\end{enumerate}
Notice that, since $\alpha_{2,K}\in K\setminus\mathbb{Q}$, the polynomial $F_{\bfz}$ always satisfies condition $(i)$ of Proposition \ref{prop : Anumber}. We start with the following:
\begin{lem}
Let $L$ be the Galois closure of $K$. Let $\lambda\in L\setminus\{0,1\}$, and let $\sigma,\tau\in\Emb (K/\mathbb{Q})\setminus\{\Id\}$ such that $\sigma\neq\tau$. Let $V_{\sigma,\tau,\lambda}$ be the $\mathbb{Q}$-vector space defined as
\[
V_{\sigma,\tau,\lambda}=\left\{x\in K:x=\frac{1}{1-\lambda}\sigma(x)-\frac{\lambda}{1-\lambda}\tau(x)\right\}.
\]
Then $\dim_{\mathbb{Q}}(V_{\sigma,\tau,\lambda})\leq\frac{2n}{3}$.
\label{lem : subspace}
\end{lem}
\begin{proof}
Let $L$ be the Galois closure of $K$, observe that we can extend any embedding in $\Emb(K/\mathbb{Q})$ to an element of $\Gal(L/\mathbb{Q})$: for every $\rho\in \Emb(K/\mathbb{Q})$, we choose one of these extensions to the whole $L$ and denote it $\tilde{\rho}$. Consider the linear map
\[
\varphi : K\rightarrow L,
\]
defined as $\varphi (x)=x-\frac{1}{1-\lambda}\sigma(x)+\frac{\lambda}{1-\lambda}\tau(x)$. Then $V_{\sigma,\tau,\lambda}=\text{Ker}(\varphi)$. Moreover, denote by $W_{\sigma,\tau,\lambda}=\text{Im}(\varphi)$. Let $d=\dim V_{\sigma,\tau,\lambda}$. Then $\dim_{\mathbb{Q}} W_{\sigma,\tau,\lambda}=n-d$, and let $\{h_{1},...,h_{n-d}\}$ be a basis of $W_{\sigma,\tau,\lambda}$. Let $\gamma_{1},...,\gamma_{n-d+1}\in\Emb(K/\mathbb{Q})$ be distinct morphisms and consider the system
\[
\begin{cases}
    X_{1}\tilde{\gamma}_{1}(h_{1})+\cdots +X_{n-d+1}\tilde{\gamma}_{n-d+1}(h_{1})=0\\
    X_{1}\tilde{\gamma}_{1}(h_{2})+\cdots +X_{n-d+1}\tilde{\gamma}_{n-d+1}(h_{2})=0\\
    \qquad\vdots\qquad\qquad\qquad\qquad\vdots\\X_{1}\tilde{\gamma}_{1}(h_{n-d})+\cdots +X_{n-d+1}\tilde{\gamma}_{n-d+1}(h_{n-d})=0\\
    \end{cases}
\]
Since this system has $n-d$ equations in $n-d+1$ variables, one of the columns has to be a linear combination of the others. Hence, without loss of generality:
\[\mu_{1}\begin{pmatrix}\tilde{\gamma}_{1}(h_{1})\\ \tilde{\gamma}_{1}(h_{2})\\\vdots\\ \tilde{\gamma}_{1}(h_{n-d})\end{pmatrix}+\cdots+\mu_{n-d}\begin{pmatrix}\tilde{\gamma}_{n-d}(h_{1})\\ \tilde{\gamma}_{n-d}(h_{2})\\\vdots\\ \tilde{\gamma}_{n-d}(h_{n-d})
\end{pmatrix}+\mu_{n-d+1}\begin{pmatrix}\tilde{\gamma}_{n-d+1}(h_{1})\\ \tilde{\gamma}_{n-d+1}(h_{2})\\\vdots\\ \tilde{\gamma}_{n-d+1}(h_{n-d})\end{pmatrix}=\begin{pmatrix}0\\ 0\\\vdots\\ 0\end{pmatrix}
\]
for some $(\mu_{1},...,\mu_{n-d+1})\neq\boldsymbol{0}$. Knowing that $h_{1},...,h_{n-d}$ is a basis of $W_{\sigma,\tau,\lambda}$ and that $\varphi$ vanishes on $V_{\sigma,\tau,\lambda}$, one has
\[
\mu_{1}(\tilde{\gamma}_{1}\circ\varphi)+\cdots \mu_{n-d+1}(\tilde{\gamma}_{n-d+1}\circ\varphi)=0.
\]
By contradiction, assume that $d>\frac{2n}{3}$. Then, $n-d < \frac{n}{3}$. We claim that for every $k< \frac{n}{3}$, we can find $A_{k}=\{\text{Id}=\rho_{1},\rho_{2},...,\rho_{k+1}\}\subset\Emb(K/\mathbb{Q})$, such that:
\begin{equation}
\rho_{i}\neq\rho_{j},\quad \rho_{i}\neq\tilde{\rho}_{j}\circ\sigma, \quad \rho_{i}\neq\tilde{\rho}_{j}\circ\tau\quad\text{for any }j<i.
\label{eq : i,j}
\end{equation}
We do this by induction, when $k=0$, we can take $\{\text{Id}\}$. Assume the claim is true for $k-1$, then the set 
\[
B_{k-1}=\bigcup_{\rho\in A_{k-1}}\{\rho,\tilde{\rho}\circ\sigma,\tilde{\rho}\circ\tau\},
\]
has cardinality $\#B_{k-1}\leq 3k<n$ (we are assuming $k<\frac{n}{3}$). Since $\#\Emb(K/\mathbb{Q})=n$, we can find $\rho_{k+1}\in \Emb(K/\mathbb{Q})\setminus B_{k-1}$ which proves the claim. Since  $n-d<\frac{n}{3}$, the claim implies that there exists $A_{n-d}=\{\text{Id}=\rho_{1},...,\rho_{n-d+1}\}\subset\Emb(K/\mathbb{Q})$ such that for any $i=1,..,n-d+1$
\[
\rho_{i}\neq\rho_{j},\quad \rho_{i}\neq\tilde{\rho}_{j}\circ\sigma, \quad \rho_{i}\neq\tilde{\rho}_{j}\circ\tau\quad\text{for any }j<i.
\]
Since we have $n-d+1$ $\mathbb{Q}$-linear maps $K\rightarrow L$, by the previous discussion, we can find $(\mu_{1},...,\mu_{n-d+1})\in L^{n-d+1}\setminus\{\boldsymbol{0}\}$ such that
\[
\mu_{1}(\tilde{\rho}_{1}\circ\varphi)+\cdots \mu_{n-d+1}(\tilde{\rho}_{n-d+1}\circ\varphi)=0.
\]
Let $\ell$ be the largest integer such that $\mu_{\ell}\neq 0$. Then
\[
\tilde{\rho}_{\ell}\circ\varphi=-\mu_{\ell}^{-1}\sum_{i=1}^{ \ell-1}\mu_{i}(\tilde{\rho}_{i}\circ\varphi)
\]

On the other hand, using the definition of $\varphi$, one would get that
\begin{equation}
\begin{split}
\rho_{\ell}&=+\frac{1}{1-\tilde{\rho}_{\ell}(\lambda)}(\tilde{\rho}_{\ell}\circ\sigma)-\frac{\tilde{\rho_{\ell}}(\lambda)}{1-\tilde{\rho}_{\ell}(\lambda)}(\tilde{\rho}_{\ell}\circ\tau)\\&\quad-\mu_{\ell}^{-1}\sum_{i=1}^{\ell-1}\mu_{i}\left(\rho_{i}-\frac{1}{1-\tilde{\rho}_{i}(\lambda)}(\tilde{\rho}_{i}\circ\sigma)+\frac{\tilde{\rho_{i}}(\lambda)}{1-\tilde{\rho_{i}}(\lambda)}(\tilde{\rho}_{i}\circ\tau)\right),
\label{eq : char}
\end{split}
\end{equation}
which would contradict the Artin-Dedekind linear independence of characters (\cite[Page $76$, Lemma $76$]{Rot98}) since, by construction, the character $\rho_{\ell}$ does not appear in the right-hand side of $(\ref{eq : char})$. 
\end{proof}

\begin{lem}
    Let $K/\mathbb{Q}$ be a number field of even degree $n\geq 2$ and $L$ its Galois closure. Moreover, write $\Emb(K/\mathbb{Q})=\{\sigma_{1}=\Id,\sigma_{2},\ldots,\sigma_{n}\}$. Let $g\in S_{n}$ be a fixed point free involution and denote $I_{g}\subset\{1,...,n\}$ such that $1\in I_{g}$, $\#I_{g}=\frac{n}{2}$, and for all $i\in I_{g}$, $g(i)\not\in I_{g}$. Then for every $\boldsymbol{\lambda}=(\lambda_{i})_{i\in I_{g}}\in L^{\frac{n}{2}}\setminus\{\boldsymbol{0}\}$, 
    \[
    K_{g,\boldsymbol{\lambda}}=\{x\in K: \lambda_{i}(\sigma_{1}(x)-\sigma_{g(1)}(x))=\lambda_{1}(\sigma_{i}(x)-\sigma_{g(i)}(x))\text{ for all }i\in I_{g}\setminus\{1\}\}
    \]
    is a vector space with $\dim_{\mathbb{Q}}K_{g,\boldsymbol{\lambda}}\leq \frac{n}{2}+1$.
    \label{lem : artin2}
\end{lem}
\begin{proof}
    Clearly $K_{g,\boldsymbol{\lambda}}$ is a vector space since all the $\Emb(K/\mathbb{Q})$ are linear maps. We first deal with the case when $\lambda_{i}=0$ for some $i\in I_{g}$. Since $(\lambda_{i})\neq \boldsymbol{0}$, we can find $j\in I_{g}$ such that
     \[
    K_{g,\boldsymbol{\lambda}}\subset \{x\in K: \sigma_{j}(x)=\sigma_{g(j)}(x)\}.
    \]
    Since $\{x\in K: \sigma_{j}(x)=\sigma_{g(j)}(x)\}$ is a subfield of $K$, 
    \[
   \dim_{\mathbb{Q}} K_{g,\boldsymbol{\lambda}}\leq \dim_{\mathbb{Q}}\{x\in K: \sigma_{j}(x)=\sigma_{g(j)}(x)\}\leq \frac{n}{2}.
    \]
    Hence we may assume that $\lambda_{i}\neq0$ for all $i\in I_{g}$. Observe that $K_{g,\boldsymbol{\lambda}}$ is defined using $\frac{n}{2}-1$ conditions. Suppose, for a contradiction, that $\dim_{\mathbb{Q}}K_{g,\boldsymbol{\lambda}}> \frac{n}{2}+1$ then we can find $(\mu_{i})_{i\in I_{g}\setminus\{1\}}\in L^{\frac{n}{2}-1}\setminus\{\boldsymbol{0}\}$ such that
    \[
    \begin{split}
    0&=\sum_{i\in I_{g}\setminus\{1\}}\mu_{i}(\lambda_{i}(\sigma_{1}(x)-\sigma_{g(1)}(x))-\lambda_{1}(\sigma_{i}(x)-\sigma_{g(i)}(x)))\\&=\left( \sum_{i\in I_{g}\setminus\{1\}}\lambda_{i}\mu_{i}\right)(\sigma_{1}(x)-\sigma_{g(1)}(x)) - \sum_{i\in I_{g}\setminus\{1\}}\lambda_{1}\mu_{i}(\sigma_{i}(x)-\sigma_{g(i)}(x))
    \end{split}
    \]
    which contradicts Artin's independence of characters.
\end{proof}

We recall that, for every $\bfz\in\mathbb{Z}^{n-2}$, we have $z=z_{3}\alpha_{3,K}+\cdots+z_{n}\alpha_{n,K}\in K$. Moreover, in what follows we denote
\[
E_{K,3}= \sum_{3\leq j\leq\frac{2n}{3}}\left(\frac{1}{n}-\delta_{j,K}\right).
\]
\begin{cor}
Let $K/\mathbb{Q}$ be a number field of degree $n\geq 5$ which does not contain any subfield of index $2$. Let $\sigma,\tau\in\Emb (K/\mathbb{Q})\setminus\{\Id\}$ be two distinct elements in $\Emb(K/\mathbb{Q})$ such that $\tau(\alpha_{2,K})\neq\alpha_{2,K}$. Then
\[
\#\left\{\bfz\in\mathcal{C}\cap \mathbb{Z}^{n-2}:\sigma(z)-z= \frac{\sigma(\alpha_{2,K})-\alpha_{2,K}}{\tau(\alpha_{2,K})-\alpha_{2,K}}\cdot(\tau(z)-z)\right\}\ll_{n} D_{K}^{E_{K,3}}
\]
\label{cor : subspaceF}
\end{cor}
\begin{proof}
In what follows for every $\rho\in\Emb (K/\mathbb{Q})$, we denote by
\[
K^{\rho}=\{x\in K:\rho(x)=x\}.
\]
Consider the vector space
\[
K_{\sigma,\tau}=\left\{x\in K:\sigma(x)-x= \frac{\sigma(\alpha_{2,K})-\alpha_{2,K}}{\tau(\alpha_{2,K})-\alpha_{2,K}}\cdot(\tau(x)-x)\right\}.
\]
We first handle the case when $\frac{\sigma(\alpha_{2,K})-\alpha_{2,K}}{\tau(\alpha_{2,K})-\alpha_{2,K}}=1$. In this case, $x$ satisfies $\sigma(x)=\tau(x)$, i.e. \[
K_{\sigma,\tau}=\{x\in K: \sigma(x)=\tau(x) \}
\]
which is a $\mathbb{Q}$-vector space of dimension at most $\frac{n}{3}$ (since we are assuming that $K$ does not contain any subfield of index $2$). Since $1\in K_{\sigma,\tau}$, we have
\[
\dim\left(\left\{(z_{3},...,z_{n}):z_{3}\alpha_{3,K}+\cdots+z_{n}\alpha_{n,K}\in K_{\sigma,\tau}\right\}\right)\leq \frac{n}{3}-1.
\]
Thus,
\[
\#\{\bfz\in\mathcal{C}\cap\mathbb{Z}^{n-2}:z\in K_{\sigma,\tau}\}\ll_{n}D_{K}^{\sum_{3\leq j\leq\frac{n}{3}+1}\left(\frac{1}{n}-\delta_{j,K}\right)}\ll_{n} D_{K}^{E_{K,3}},
\]
since when $n\geq 5$, $\frac{n}{3}+1\leq \frac{2n}{3}$. Similarly, if $\sigma(\alpha_{2,K})=\alpha_{2,K}$, it follows that $z\in K^{\sigma}$ which is again a $\mathbb{Q}$-vectorspace of dimension at most $\frac{n}{3}$. Thus, we get again
\[
\#\{\bfz\in\mathcal{C}\cap\mathbb{Z}^{n-2}:z\in K_{\sigma,\tau}\}\ll_{n} D_{K}^{E_{K,3}},
\]

We now assume $\frac{\sigma(\alpha_{2,K})-\alpha_{2,K}}{\tau(\alpha_{2,K})-\alpha_{2,K}}\neq 1 $. In this case, we can apply Lemma $\ref{lem : subspace}$ with $\lambda=\frac{\sigma(\alpha_{2,K})-\alpha_{2,K}}{\tau(\alpha_{2,K})-\alpha_{2,K}}\in L\setminus\{0,1\}$, getting $\dim_{\mathbb{Q}} K_{\sigma,\tau}\leq \frac{2n}{3}$. Moreover, one has that $1,\alpha_{2,K}\in K_{\sigma,\tau}$, thus
\[
\dim\left(\left\{(x_{3},...,x_{n}):x_{3}\alpha_{3,K}+\cdots+x_{n}\alpha_{n,K}\in K_{\sigma,\tau}\right\}\right)\leq \frac{2n}{3}-2.
\]
Thus, it follows that
\[
\#\{\bfz\in\mathcal{C}\cap \mathbb{Z}^{n-2}:z_{3}\alpha_{3,K}+\cdots+z_{n}\alpha_{n,K}\in K_{\sigma,\tau}\}\ll D_{K}^{E_{K,3}},
\]
as claimed.
\end{proof}
\begin{lem}
    Under the notation and hypotheses of Corollary $\ref{cor : subspaceF}$, we have
    \[
    \#\{\bfz\in\mathcal{C}\cap\mathbb{Z}^{n-2}:\sigma(z)=z,\text{ for some }\sigma\in\Emb(K/\mathbb{Q})\setminus\{\Id\}\}\ll D_{K}^{E_{K,3}}
    \]
    \label{lem : subfield}
\end{lem}
\begin{proof}
One has that
    \[
    \{x\in K:\sigma(x)=x,\text{ for some }\sigma\in\Emb(K/\mathbb{Q})\setminus\{\Id\}\}=\bigcup_{\substack{\sigma\in\Emb(K/\mathbb{Q})\\\sigma\neq \text{Id}}}K^{\sigma},
    \]
    where $K^{\sigma}$ is the field of all the elements in $K$ which are fixed by $\sigma$. Since $[K:K^{\sigma}]\geq 3$ (we are assuming that $K$ does not contain any intermediate field of index $2$), then $K^{\sigma}$ is a $\mathbb{Q}$-vector space of dimension $\leq\frac{n}{3}$, for every $\sigma\in \Emb(K/\mathbb{Q})\setminus\{\Id\}$. Moreover, since $1\in K^{\sigma}$
    \[
\dim\left(\left\{(z_{3},...,z_{n}):z_{3}\alpha_{3,K}+\cdots+z_{n}\alpha_{n,K}\in K^{\sigma}\right\}\right)\leq\frac{n}{3}-1.
\]
Hence,
\[
\begin{split}
    \#\{\bfz\in\mathcal{C}\cap\mathbb{Z}^{n-2}:\sigma(z)=z,\text{ for some }\sigma\in\Emb(K/\mathbb{Q})\setminus\{\Id\}\}&\ll D_{K}^{\sum_{3\leq j\leq\frac{n}{3}+1}\left(\frac{1}{n}-\delta_{j,K}\right)}\\&\ll D_{K}^{\sum_{3\leq j\leq\frac{2n}{3}}\left(\frac{1}{n}-\delta_{j,K}\right)},
    \end{split}
    \]
since when $n\geq 5$, $\frac{n}{3}+1\leq \frac{2n}{3}$.
\end{proof}
\begin{lem}
    Let $K/\mathbb{Q}$ be a number field of even degree $n\geq 6$. Let $g\in S_{n}$ be a fixed point free involution and denote $I_{g}\subset\{1,...,n\}$ such that $1\in I_{g}$, $\# I_{g}=\frac{n}{2}$ and for all $i\in I_{g}$, $g(i)\not\in I_{g}$. Then
    \[
    \#\{\bfz\in\mathcal{C}\cap \mathbb{Z}^{n-2}:\exists y\in \overline{\mathbb{Q}}\text{ such that } \sigma_{i}(\alpha_{2,K})y+\sigma_{i}(z)=\sigma_{g(i)}(\alpha_{2,K})y+\sigma_{g(i)}(z)\quad \forall i\in I_{g}\}\ll D_{K}^{E_{K,3}}
    \]
\label{lem : secondcount}    
\end{lem}
\begin{proof}
    In what follows we denote $\sigma_{1}=\Id$. For every $i\in I_{g}$, denote $\lambda_{i}=\sigma_{i}(\alpha_{2,K})-\sigma_{g(i)}(\alpha_{2,K})$. Observe that, for every $x\in K$, if there exists $y\in \overline{\mathbb{Q}}$ such that
     \[
     \sigma_{i}(\alpha_{2,K})y+\sigma_{i}(x)=\sigma_{g(i)}(\alpha_{2,K})y+\sigma_{g(i)}(x)\quad \forall i\in I_{g}
     \]
     then either $y=-\lambda_{i}^{-1}(\sigma_{i}(x)-\sigma_{g(i)}(x))$ if $\lambda_{i}\neq 0$, or $\sigma_{i}(x)=\sigma_{g(i)}(x)$, hence such $x$ belongs to the vector space
    \[
    K_{g,\boldsymbol{\lambda}}=\{x\in K : \lambda_{i}(\sigma_{1} (x)-\sigma_{g(1)} (x))=\lambda_{1}(\sigma_{i} (x)-\sigma_{g(i)} (x))\text{ for all }i\in I_{g}\setminus\{1\}\}.
    \]
    We first deal with the case when $\lambda_{1}\neq 0$. Since $\dim_{\mathbb{Q}}K_{g,\boldsymbol{\lambda}}\leq \frac{n}{2}+1$ thanks to Lemma $\ref{lem : artin2}$, and $1,\alpha_{2,K}\in K_{g,\boldsymbol{\lambda}}$, it follows that
    \[
    \dim\{\bfz\in \mathbb{Q}^{n-2}: z_{3}\alpha_{3,K}+\cdots+z_{n}\alpha_{n,K}\in K_{g,\boldsymbol{\lambda}}\}\leq \frac{n}{2}-1.
    \]
    Using the fact that $\frac{n}{2}-1\leq \frac{2n}{3}-2$ when $n\geq 6$, the result follows. If $\lambda_{1}=\alpha_{2,K}-\sigma_{g(1)}(\alpha_{2,K})=0$, then for every $x\in K$ for which there exists $y\in \overline{\mathbb{Q}}$ 
     \[
     \alpha_{2,K}y+x=\sigma_{g(1)}(\alpha_{2,K})y+\sigma_{g(1)}(x)
     \]
    we deduce that $x=\sigma_{g(1)} (x)$, i.e. $x\in K^{\sigma_{g(1)}}$. Since $\dim_{\mathbb{Q}} K^{\sigma_{g(1)}}\leq \frac{n}{2}$, and $1,\alpha_{2,K}\in K^{\sigma_{g(1)}}$, it follows that
    \[
    \dim\{\bfz\in \mathbb{Q}^{n-2}: z_{3}\alpha_{3,K}+\cdots+z_{n}\alpha_{n,K}\in K^{\sigma_{g(1)}}\}\leq \frac{n}{2}-2,
    \]
    which is again acceptable.
\end{proof}

We are finally ready to prove
\begin{cor}
   Assume that $K/\mathbb{Q}$ does not contain a subfield of index $2$, and $\delta_{2,K}\geq\frac{8}{21n}$. Then
    \[
    \sum_{\substack{\bfz\in\mathcal{C}\\ \bfz\in Z_{F}\setminus Z_{\square} }}\sum_{\substack{(x_{1},x_{2})\in\mathcal{B}\cap\mathbb{Z}^{2}\\F_{\bfz}(x_{1},x_{2})=\square}}1\ll_{n,\varepsilon}  D_{K}^{\frac{1}{2}-\frac{1}{2n}-\frac{1}{7n}+\varepsilon}
    \]
    \label{cor : Ffail}
\end{cor}
\begin{proof}
Take
\[
T_{F}=\left(\bigcup_{\substack{\sigma,\tau\in\Emb(K/\mathbb{Q})\setminus\{\Id\}\\\sigma\neq \tau\\\tau (\alpha_{2,K})\neq\alpha_{2,K}}}K_{\sigma,\tau}\right)\cup\left(\bigcup_{\substack{g\in S_{n}\\g^{2}=1\\g(i)\neq i\forall i}}K_{g}\right)\cup\left(\bigcup_{\substack{\sigma\in\Emb(K/\mathbb{Q})\setminus\{\Id\}\\\sigma\neq \Id}}K^{\sigma}\right),
\]
where for all $\sigma,\tau\in\Emb(K/\mathbb{Q})$ , $ \tau(\alpha_{2,K})\neq\alpha_{2,K}$, and for every $g\in S_{n}$ fixed point free involution we denote
\[
\begin{split}
&K_{\sigma,\tau}=\left\{x\in K:\sigma(x)-x= \frac{\sigma(\alpha_{2,K})-\alpha_{2,K}}{\tau(\alpha_{2,K})-\alpha_{2,K}}\cdot(\tau(x)-x)\right\}\\& K_{g}=\{x\in K:\exists y\in \overline{\mathbb{Q}}\text{ such that } \sigma_{i}(\alpha_{2,K})y+\sigma_{i}(x)=\sigma_{g(i)}(\alpha_{2,K})y+\sigma_{g(i)}(x)\quad \forall i\in I_{g}\}\\& K^{\sigma}=\{x\in K: \sigma(x)=x\},
\end{split}
\]
and where $I_{g}\subset\{1,...,n\}$ such that $1\in I_{g}$, $\# I_{g}=\frac{n}{2}$ and for all $i\in I_{g}$, $g(i)\not\in I_{g}$.
Then $Z_{F}=\{\bfz\in\mathbb{Z}^{n-2}:z_{3}\alpha_{3,K}+\cdots+z_{n}\alpha_{n,K}\in T_{F}\}$.
Moreover, we recall that
\[
E_{K,3}= \sum_{3\leq j\leq\frac{2n}{3}}\left(\frac{1}{n}-\delta_{j,K}\right).
\]
We start by observing that
\[
\begin{split}
\sum_{\substack{\bfz\in\mathcal{C}\\ \bfz\in Z_{F}\setminus Z_{\square} }}\sum_{\substack{(x_{1},x_{2})\in\mathcal{B}\cap\mathbb{Z}^{2}\\F_{\bfz}(x_{1},x_{2})=\square}}1&= \sum_{\substack{\bfz\in\mathcal{C}\cap Z_{F} \\F_{\bfz}(X_{1},X_{2})\neq\square}}\sum_{\substack{(x_{1},x_{2})\in\mathcal{B}\cap\mathbb{Z}^{2}\\F_{\bfz}(x_{1},x_{2})=\square}}1\\&=
\sum_{\substack{\bfz\in\mathcal{C}\cap Z_{F}  \\F_{\bfz}(X_{1},X_{2})\neq\square}}\quad\sum_{|x_{2}|\leq D_{K}^{\frac{1}{n}-\delta_{2,K}}}\sum_{\substack{|x_{1}|\leq D_{K}^{\frac{1}{n}}\\F_{\bfz}(x_{1},x_{2})=\square}}1\\& =\sum_{\substack{\bfz\in\mathcal{C}\cap Z_{F}  \\F_{\bfz}(X_{1},X_{2})\neq\square}}\quad\sum_{\substack{|x_{2}|\leq D_{K}^{\frac{1}{n}-\delta_{2,K}}\\ F_{\bfz}(X_{1},x_{2})\neq\square}}\sum_{\substack{|x_{1}|\leq D_{K}^{\frac{1}{n}}\\F_{\bfz}(x_{1},x_{2})=\square}}1\\&\quad+\sum_{\substack{\bfz\in\mathcal{C}\cap Z_{F}  \\F_{\bfz}(X_{1},X_{2})\neq\square}}\quad\sum_{\substack{|x_{2}|\leq D_{K}^{\frac{1}{n}-\delta_{2,K}}\\ F_{\bfz}(X_{1},x_{2})=\square}}\sum_{\substack{|x_{1}|\leq D_{K}^{\frac{1}{n}}\\F_{\bfz}(x_{1},x_{2})=\square}}1.
\end{split}
\]
For the first term, we may use the Bombieri-Pila bound (\cite{BP89}), getting 
\[
\sum_{\substack{\bfz\in\mathcal{C}\cap Z_{F}  \\F_{\bfz}(X_{1},X_{2})\neq\square}}\quad\sum_{\substack{|x_{2}|\leq D_{K}^{\frac{1}{n}-\delta_{2,K}}\\ F_{\bfz}(X_{1},x_{2})\neq\square}}\sum_{\substack{|x_{1}|\leq D_{K}^{\frac{1}{n}}\\F_{\bfz}(x_{1},x_{2})=\square}}1\ll_{n,\varepsilon} \# (\mathcal{C}\cap Z_{F}) D_{K}^{\frac{1}{n}-\delta_{2,K}} D_{K}^{\frac{1}{2n}+\varepsilon}
\]
We focus on the second sum: In this case an application of \cite[Lemma $4.1$]{BSTTZ20} with $\beta=x_{1}+x_{2}\alpha_{2,K}+z$ gives
\[
\sum_{\substack{\bfz\in\mathcal{C}\cap Z_{F}  \\F_{\bfz}(X_{1},X_{2})\neq\square}}\quad\sum_{\substack{|x_{2}|\leq D_{K}^{\frac{1}{n}-\delta_{2,K}}\\ F_{\bfz}(X_{1},x_{2})=\square}}\sum_{\substack{|x_{1}|\leq D_{K}^{\frac{1}{n}}\\F_{\bfz}(x_{1},x_{2})=\square}}1\ll_{\varepsilon} D_{K}^{\frac{1}{4}+\varepsilon}.
\] 
Thus
\[
\sum_{\substack{\bfz\in\mathcal{C}\\ \bfz\in Z_{F}\setminus Z_{\square} }}\sum_{\substack{(x_{1},x_{2})\in\mathcal{B}\cap\mathbb{Z}^{2}\\F_{\bfz}(x_{1},x_{2})=\square}}1\ll D_{K}^{\frac{1}{4}+\varepsilon}+\# (\mathcal{C}\cap Z_{F}) D_{K}^{\frac{1}{n}-\delta_{2,K}+\frac{1}{2n}+\varepsilon}
\]

Using Corollary $\ref{cor : subspaceF}$, Lemma $\ref{lem : subfield}$, and Lemma $\ref{lem : secondcount}$ one gets
\[
 \# (\mathcal{C}\cap Z_{F})\ll D_{K}^{E_{K,3}}.
\]
On the other hand, since $\delta_{j,K}\geq\delta_{2,K}\geq \frac{8}{21n}$ for every $j\geq 2$, we get

\[
D_{K}^{\frac{1}{2n}+\left(\frac{1}{n}-\delta_{2,K}\right)+E_{K,3}}=D_{K}^{\frac{1}{2n}+ \sum_{2\leq j\leq\frac{2n}{3}}\left(\frac{1}{n}-\delta_{j,K}\right)}\ll D_{K}^{\frac{1}{2n}+\left(\left\lfloor\frac{2n}{3}\right\rfloor-1\right)\cdot\frac{13}{21n}}.
\]
Since $\frac{1}{2n}+\left(\left\lfloor\frac{2n}{3}\right\rfloor-1\right)\cdot\frac{13}{21n}\leq \frac{1}{2}-\frac{1}{2n}-\frac{1}{7n}$, when $n\geq 5$, the result follows.
\end{proof}
\subsection{Preparation for the sieve: smooth weights.}
As previously mentioned, our strategy is to fix $\bfz = (z_{3}, \dots, z_{n})$ and study how often
\[
F_{\bfz}(X_{1}, X_{2}) = N_K\left(X_{1} + X_{2} \alpha_{2,K} + z_{3} \alpha_{3,K} + \cdots + z_{n} \alpha_{n,K}\right)
\]
is a square, as $(x_{1}, x_{2})$ ranges over the box
\[
\mathcal{B} := \left[-D_K^{\frac{1}{n}}, D_K^{\frac{1}{n}}\right] \times \left[-D_K^{\frac{1}{n} - \delta_{2,K}}, D_K^{\frac{1}{n} - \delta_{2,K}}\right].
\]
Having this in mind, define
\[
a_{\bfz}(m) = \sum_{\substack{(x_{1}, x_{2}) \in \mathcal{B} \cap \mathbb{Z}^2 \\ F_{\bfz}(x_{1}, x_{2}) = m}} 1.
\]
Then,
\begin{equation}
\#\left\{(x_{1}, x_{2}) \in \mathcal{B}\cap\mathbb{Z}^{2} : F_{\bfz}(x_{1}, x_{2}) = \square \right\} \leq \sum_{k \in \mathbb{Z}} a_{\bfz}(k^2).
\label{eq : seq}
\end{equation}
Our goal is to estimate the right-hand side of \eqref{eq : seq} using the square sieve (including its version for composite moduli). As we will see in the upcoming sections, after applying the sieve lemma, we will need to apply the Poisson summation formula. For that reason, rather than working directly with the sequence $(a_{\bfz}(m))_{m \in \mathbb{Z}}$, we will instead introduce suitable smooth weights. We begin by choosing two smooth functions $g_{1}, g_{2} : \mathbb{R} \to \mathbb{R}_{\geq 0}$ satisfying the following properties:

\begin{itemize}
\item[$i)$] $g_{1}$ is supported on $\left[-2D_{K}^{\frac{1}{n}},2D_{K}^{\frac{1}{n}}\right]$, and $\geq 1$ in the interval $\left[-D_{K}^{\frac{1}{n}},D_{K}^{\frac{1}{n}}\right]$; the function $g_{2}$ is supported on $\left[-2D_{K}^{\frac{1}{n}-\delta_{2,K}},2D_{K}^{\frac{1}{n}-\delta_{2,K}}\right]$ and $\geq 1$ on the interval $\left[-D_{K}^{\frac{1}{n}-\delta_{2,K}},D_{K}^{\frac{1}{n}-\delta_{2,K}}\right]$
\item[$ii)$] for any $j\geq 0$, we have that
\begin{equation}
\left|\frac{\partial^{j}g_{1}}{\partial x_{1}^{j}}\right|\ll_{j} D_{K}^{-\frac{j}{n}},\quad \left|\frac{\partial^{j}g_{2}}{\partial x_{2}^{j}}\right|\ll_{j} D_{K}^{-j\left(\frac{1}{n}-\delta_{2,K}\right)}.
\label{eq : fourbound}
\end{equation}
\end{itemize}
From the condition above, one deduces the following bounds on the Fourier transforms of $g_{1}$, and $g_{2}$
\[
|\hat{g}_{1}(u)|\ll D_{K}^{\frac{1}{n}}\left(1+|u|D_{K}^{\frac{1}{n}}\right)^{-2},\quad |\hat{g}_{2}(u)|\ll D_{K}^{\frac{1}{n}-\delta_{2,K}}\left(1+|u|D_{K}^{\frac{1}{n}-\delta_{2,K}}\right)^{-2}.
\]
Then, for any $m\in\mathbb{Z}$ we define
\begin{equation}
\omega_{\bfz} (m)=\begin{cases}\sum_{\substack{(x_{1},x_{2})\in\mathbb{Z}^{2}\\F_{\bfz}(x_{1},x_{2})=m}}g_{1}(x_{1})g_{2}(x_{2})        &\text{if $m\neq 0$}\\
0               &\text{otherwise.}
\label{eq : weight}
\end{cases}
\end{equation}
We remark that when $m=0$,
\[
\sum_{\substack{(x_{1}, x_{2}) \in \mathcal{B} \cap \mathbb{Z}^2 \\ F_{\bfz}(x_{1}, x_{2}) = 0}} 1=
\begin{cases}
1   &\text{if $\bfz=\boldsymbol{0}$}\\
0                   &\text{otherwise,}
\end{cases}
\] 
since $F_{\bfz}(x_{1},x_{2})=N_{K}(x_{1}+x_{2}\alpha_{2,K}+z_{3}\alpha_{3,K}+\cdots+z_{n}\alpha_{n,K})$ and $x\in K$ is such that $N_{K}(x)=0$ if and only if $x=0$. Hence
\[
\#\left\{(x_{1}, x_{2}) \in \mathcal{B}\cap\mathbb{Z}^{2} : F_{\bfz}(x_{1}, x_{2}) = \square \right\} \leq 1+ \sum_{k \in \mathbb{Z}} \omega_{\bfz}(k^2).
\]

\subsection{The square sieve: proof of Theorem \ref{thm : deltalarge}.}
\label{sec : deltalarge}
We record the following Sieve Lemma due to Heath-Brown \cite{HB84}:
\begin{lem}
    Let $\mathcal{P}$ be a finite set of $P$ prime numbers. Let $(\omega(m))_{m\in \mathbb{Z}}$ be a sequence of nonnegative numbers such that $\omega(m)=0$ if either $m=0$, or $|m|>e^{P}$. Then
    \[
    \sum_{k}\omega(k^{2})\ll \frac{1}{P}\sum_{m}\omega(m)+\frac{1}{P^{2}}\sum_{\substack{p,q\in\mathcal{P}\\p\neq q}}\left|\sum_{m}\omega (m)\left(\frac{m}{pq}\right)\right|.
    \]
    \label{lem : sqsieve}
\end{lem}
We will use the sieve inequality in Lemma $\ref{lem : sqsieve}$ to show:
\begin{lem}
    Let $\bfz\in\mathbb{Z}^{n-2}\cap\mathcal{C}$ such that $F_{\bfz}(X_{1},X_{2})$ is not a square in $\mathbb{Z}[X_{1},X_{2}]$, and consider the sequence $(\omega_{\bfz}(m))_{m\in\mathbb{Z}}$ defined in $(\ref{eq : weight})$. If $\delta_{2,K}\leq \frac{1}{2n}$, then,
    \[
    \sum_{k}\omega_{\bfz}(k^{2})\ll_{n,\varepsilon} D_{K}^{\frac{4}{3n}-\frac{2\delta_{2,K}}{3}+\varepsilon}.
    \]
    \label{lem : sieve1}
\end{lem}
To prove Lemma \ref{lem : sieve1}, let $\mathcal{P}$ be an auxiliary set of odd primes and put $P=\#\mathcal{P}$. We assume that $P>(\log D_K)^2$ and specify the set below. Applying Lemma \ref{lem : sqsieve} to the sequence $(\omega_{\bfz}(m))_{m\in\mathbb{Z}}$ gives
 \begin{equation}
    \sum_{k}\omega_{\bfz}(k^{2})\ll \frac{1}{P}\sum_{m}\omega_{\bfz}(m)+\frac{1}{P^{2}}\sum_{\substack{p,q\in\mathcal{P}\\p\neq q}}\left|\sum_{m}\omega_{\bfz} (m)\left(\frac{m}{pq}\right)\right|,
    \label{eq : sqafter}
    \end{equation}
By the definition of the sequence $\omega_{\bfz}(m)$, it follows that
\begin{equation}
\frac{1}{P}\sum_{m}\omega_{\bfz}(m)\ll \frac{D_{K}^{\frac{2}{n}-\delta_{2,K}}}{P}.
\label{eq : sieveeasypart}
\end{equation}
Hence, it remains to bound the second term on the right-hand side of $(\ref{eq : sqafter})$. We start by observing that
\begin{equation}
\begin{split}
    S(\bfz,pq)&= \sum_{m\in\mathbb{Z}}\omega_{\bfz}(m)\left(\frac{m}{pq}\right)\\&=\sum_{(x_{1},x_{2})\in\mathbb{Z}^{2}}g_{1}(x_{1})g_{2}(x_{2})
\left(\frac{F_{\bfz}(x_{1},x_{2})}{pq}\right)\\&=\sum_{(a_{1},a_{2})\text{ }(pq)}\left(\frac{F_{\bfz}(a_{1},a_{2})}{pq}\right)\sum_{\substack{x_{1}\in\mathbb{Z}\\ x_{1}\equiv a_{1}\text{ } (pq)}}g_{1}(x_{1})\sum_{\substack{x_{2}\in\mathbb{Z}\\ x_{2}\equiv a_{2}\text{ } (pq)}}g_{2}(x_{2})\\&=\frac{1}{(pq)^{2}}\sum_{(u_{1},u_{2})\in\mathbb{Z}^{2}}\hat{g}_{1}\left(\frac{u_{1}}{pq}\right)\hat{g}_{2}\left(\frac{u_{2}}{pq}\right)T(\bfu,\bfz; pq), 
\end{split}
\label{eq : intro}
\end{equation}
where for every $r\geq 2$ square-free and $\bfv\in \mathbb{Z}^{2}$,
\[
T(\bfv,\bfz;r)=\sum_{(a_{1},a_{2})\text{ } (r)}\left(\frac{F_{\bfz}(a_{1},a_{2})}{r}\right)e_{r}(a_{1}v_{1}+a_{2}v_{2}).
\]
An application of the Chinese Remainder Theorem shows that
\[
S(\bfz,pq)=\frac{1}{(pq)^{2}}\sum_{(u_{1},u_{2})\in\mathbb{Z}^{2}}\hat{g}_{1}\left(\frac{u_{1}}{pq}\right)\hat{g}_{2}\left(\frac{u_{2}}{pq}\right)T(\overline{q}\bfu,\bfz; p)T(\overline{p}\bfu,\bfz; q).
\]
\subsubsection{The exponential sum $T(\bfv,\bfz; r)$.}
In the next section (see Lemma $\ref{lem : expnotshift}$), we shall prove the following
\begin{lem}
Let $\bfz=(z_{3},...,z_{n})\in\mathbb{Z}^{n-2}$, such that $F_{\bfz}(X_{1},X_{2})$ is not a square over $\mathbb{Z}$, and let $G_{\bfz}$ be the square-free part of $F_{\bfz}$. Then for every odd prime number $p$ such that
\begin{itemize}
    \item[$i)$] the polynomial $G_{\bfz}(X_{1},X_{2})\mod p$ is square-free;
    \item[$ii)$] $p\nmid \disc(f_{\alpha_{2}})$, where $f_{\alpha_{2}}$ is the minimal polynomial of $\alpha_{2}$ over $\mathbb{Q}.$
\end{itemize}
\[
|T(\bfv,\bfz;p)|\ll_{n} p(p,v_{1},v_{2})^{\frac{1}{2}}.
\]
\label{lem : exp}
\end{lem}
Lemma $\ref{lem : exp}$ is proved in Section \ref{sec : exp}. As a consequence of Lemma $\ref{lem : exp}$, we obtain
\begin{cor}
Let $\bfz=(z_{3},...,z_{n})\in\mathbb{Z}^{n-2}$, such that $F_{\bfz}(X_{1},X_{2})$ is not a square over $\mathbb{Z}$, and let $G_{\bfz}$ be the square-free part of $F_{\bfz}$. Let $Q>0$ and let $p,q$ be two distinct odd primes such that $Q<p,q<2Q$ and assume
\begin{itemize}
    \item[$i)$] the polynomials $G_{\bfz}(X_{1},X_{2})\mod p$, $G_{\bfz}(X_{1},X_{2})\mod q$ are square-free in $\mathbb{F}_{p}[X_{1},X_{2}]$ and $\mathbb{F}_{q}[X_{1},X_{2}]$ respectively;
    \item[$ii)$] $p,q\nmid \disc(f_{\alpha_{2}})$, where $f_{\alpha_{2}}$ is the minimal polynomial of $\alpha_{2}$ over $\mathbb{Q}$.
\end{itemize}
Then 
\[
S(\bfz,pq)\ll_{n} Q^{2}+ \frac{D_{K}^{\frac{2}{n}-\delta_{2,K}}}{Q},
\]
provided that  $D_{K}^{\frac{1}{2n}}\leq Q\leq D_{K}^{\frac{1}{n}-\delta_{2,K}} $.
\label{cor : sieve1}
\end{cor}
\begin{proof}
Since the bound for $T(\bfu,\bfz,pq)$ depends on the G.C.D. between $u_{1},u_{2}$ and $pq$ we start by rewriting $S(\bfz,pq)$ as
\[
S(\bfz,pq)=\frac{1}{(pq)^{2}}\sum_{\substack{\bfu\in\mathbb{Z}^{2}\\(\bfu,pq)=1}}\hat{g}_{1}\left(\frac{u_{1}}{pq}\right)\hat{g}_{2}\left(\frac{u_{2}}{pq}\right)T(\bfu,\bfz; pq)+\frac{1}{(pq)^{2}}\sum_{\substack{\bfu\in\mathbb{Z}^{2}\\(\bfu,pq)>1}}\hat{g}_{1}\left(\frac{u_{1}}{pq}\right)\hat{g}_{2}\left(\frac{u_{2}}{pq}\right)T(\bfu,\bfz; pq)
\]

From $(\ref{eq : fourbound})$, it follows that
\[
\left|\hat{g}_{1}\left(\frac{u_{1}}{pq}\right)\hat{g}_{2}\left(\frac{u_{2}}{pq}\right)\right|\ll D_{K}^{\frac{2}{n}-\delta_{2,K}} \left(1+\frac{D_{K}^{\frac{1}{n}}|u_{1}|}{pq}\right)^{-2}\cdot\left(1+\frac{D_{K}^{\frac{1}{n}-\delta_{2,K}}|u_{2}|}{pq}\right)^{-2}.
\]
On the other hand, by applying Lemma $\ref{lem : exp}$, one gets
\[
\begin{split}
\frac{1}{(pq)^{2}}\sum_{\substack{\bfu\in\mathbb{Z}^{2}\\(\bfu,pq)=1}}|T(\bfu,\bfz; pq)|&\cdot\left|\hat{g}_{1}\left(\frac{u_{1}}{pq}\right)\hat{g}_{2}\left(\frac{u_{2}}{pq}\right)\right|\ll_{n}\\&\ll_{n} pq\cdot\frac{D_{K}^{\frac{2}{n}-\delta_{2,K}}}{(pq)^{2}}\sum_{\substack{\bfu\in\mathbb{Z}^{2}}}\left(1+\frac{D_{K}^{\frac{1}{n}}|u_{1}|}{pq}\right)^{-2}\cdot\left(1+\frac{D_{K}^{\frac{1}{n}-\delta_{2,K}}|u_{2}|}{pq}\right)^{-2}\\&\ll Q^{2},
\end{split}
\]
provided that $D_{K}^{\frac{1}{n}}\leq pq\leq 4Q^{2}$. We now deal with the contributions of those $\bfu$ for which $(\bfu,pq)>1$:
\[
\begin{split}
\frac{1}{(pq)^{2}}\sum_{\substack{\bfu\in\mathbb{Z}^{2}\\(\bfu,pq)>1}}\left|\hat{g}_{1}\left(\frac{u_{1}}{pq}\right)\hat{g}_{2}\left(\frac{u_{2}}{pq}\right)\right||T(\bfu,\bfz; pq)|&\ll_{n} (pq)^{\frac{3}{2}}\cdot\frac{1}{(pq)^{2}}\sum_{\substack{\bfu\in\mathbb{Z}^{2}\\ p|\bfu}}\left|\hat{g}_{1}\left(\frac{u_{1}}{pq}\right)\hat{g}_{2}\left(\frac{u_{2}}{pq}\right)\right|\\&\quad +(pq)^{\frac{3}{2}}\cdot\frac{1}{(pq)^{2}}\sum_{\substack{\bfu\in\mathbb{Z}^{2}\\ q|\bfu}}\left|\hat{g}_{1}\left(\frac{u_{1}}{pq}\right)\hat{g}_{2}\left(\frac{u_{2}}{pq}\right)\right|.
\end{split}
\]
Provided that $q\ll D_{K}^{\frac{1}{n}-\delta_{2,K}}$, one has
\[
\begin{split}
\sum_{\substack{\bfu\in\mathbb{Z}^{2}\\ q|\bfu}}\left|\hat{g}_{1}\left(\frac{u_{1}}{pq}\right)\hat{g}_{2}\left(\frac{u_{2}}{pq}\right)\right|&\ll D_{K}^{\frac{2}{n}-\delta_{2,K}}\sum_{\substack{\bfu\in\mathbb{Z}^{2}\\ q|\bfu}}\left(1+\frac{D_{K}^{\frac{1}{n}}|u_{1}|}{pq}\right)^{-2}\cdot\left(1+\frac{D_{K}^{\frac{1}{n}-\delta_{2,K}}|u_{2}|}{pq}\right)^{-2}\\&\ll D_{K}^{\frac{2}{n}-\delta_{2,K}},
\end{split}
\]
which is enough to complete the proof of the Corollary.
\end{proof}
We are finally ready to complete the proof of Lemma $\ref{lem : sieve1}$: let $Q$ be such that $D_{K}^{\frac{1}{2n}}\leq Q\leq D_{K}^{\frac{1}{n}-\delta_{2,K}} $ and take 
\[
\begin{split}
    \mathcal{P}=\left\{p:\begin{matrix}Q\leq p\leq 2Q\\G_{\bfz}(X_{1},X_{2})\mod p\text{ is square-free in }\mathbb{F}_{p}[X_{1},X_{2}]\\p\nmid \disc(f_{\alpha_{2}}) \end{matrix}\right\},
\end{split}
\]
where $G_{\bfz}$ is the square-free part of $F_{\bfz}$, and $f_{\alpha_{2}}$ is the minimal polynomial of $\alpha_{2}$ over $\mathbb{Q}$. Only $O_n(\log D_K)$ primes fail the displayed conditions. Thus we get $P=\#\mathcal{P}\gg_n Q/\log Q$. Since $Q\geq D_K^{1/(2n)}$, this implies $P>(\log D_K)^2$ for sufficiently large $D_K$.
Then combining $(\ref{eq : sqafter})$, $(\ref{eq : sieveeasypart})$, and Corollary $\ref{cor : sieve1}$, one gets
\[
\sum_{k\in\mathbb{Z}}\omega_{\bfz} (k^{2})\ll_{n} \frac{D_{K}^{\frac{2}{n}-\delta_{2,K}}\log Q}{Q}+ Q^{2}.
\]
The optimal choice here is when $Q=D_{K}^{\frac{2}{3n}-\frac{\delta_{2,K}}{3}}$. So we get
\[
    \sum_{k\in\mathbb{Z}}\omega_{\bfz} (k^{2})\ll_{n,\varepsilon} D_{K}^{\frac{4}{3n}-\frac{2\delta_{2,K}}{3}+\varepsilon},
\]
provided that $D_{K}^{\frac{1}{n}}\leq D_{K}^{\frac{4}{3n}-\frac{2\delta_{2,K}}{3}}$ and $ D_{K}^{\frac{2}{3n}-\frac{\delta_{2,K}}{3}}\leq D_{K}^{\frac{1}{n}-\delta_{2,K}}$, which is satisfied as soon as $\delta_{2,K}\leq \frac{1}{2n}$.

\subsubsection{Proof of Theorem $\ref{thm : deltalarge}$}
We can now prove Theorem $\ref{thm : deltalarge}$. If $\delta_{2,K}\geq \frac{1}{2n}$ the result follows from \cite{BSTTZ20}. Hence we may assume $\delta_{2,K}\leq \frac{1}{2n}$. Moreover, thanks to Lemma $\ref{lem : index2}$, we may assume that $K$ does not contain a subfield of index $2$. We recall that 
\[\mathcal{C}=\left[-D_{K}^{\frac{1}{n}-\delta_{3,K}},D_{K}^{\frac{1}{n}-\delta_{3,K}}\right]\times\cdots\times\left[-D_{K}^{\frac{1}{n}-\delta_{n,K}},D_{K}^{\frac{1}{n}-\delta_{n,K}}\right],
\]
and that
\[
Z_{\square}=\{\bfz\in\mathbb{Z}^{n-2}: F_{\bfz}(X_{1},X_{2})\text{ is a square}\}.
\]
Then
\begin{equation}
\begin{split}
h_{2}(K)&\ll\# \left\{(y,z_{1},...,z_{n})\in\mathbb{Z}^{n+1}:\begin{matrix}|z_{i}|\ll D_{K}^{\frac{1}{n}-\delta_{i,K}},\quad\text{for }i=1,...,n\\y^{2}=N_{K}(z_{1}+z_{2}\alpha_{2,K}+\cdots+z_{n}\alpha_{n,K})\end{matrix}     \right\}\\&\ll\sum_{\bfz\in\mathcal{C}\cap\mathbb{Z}^{n-2}}\sum_{\substack{(x_{1},x_{2})\in\mathcal{B}\cap\mathbb{Z}^{2}\\F_{\bfz}(x_{1},x_{2})=\square}}1\\&= \sum_{\substack{\bfz\in\mathcal{C}\cap Z_{\square}}}\sum_{\substack{(x_{1},x_{2})\in\mathcal{B}\cap\mathbb{Z}^{2}\\F_{\bfz}(x_{1},x_{2})=\square}}1+\sum_{\substack{\bfz\in\mathcal{C}\cap\mathbb{Z}^{n-2}\\\bfz\not\in Z_{\square}}}\sum_{\substack{(x_{1},x_{2})\in\mathcal{B}\cap\mathbb{Z}^{2}\\F_{\bfz}(x_{1},x_{2})=\square}}1.
\end{split}
\label{eq : tog}
\end{equation}
For the second term in $(\ref{eq : tog})$, we can use Lemma $\ref{lem : sieve1}$, getting
\[
\begin{split}
\sum_{\substack{\bfz\in\mathcal{C}\cap\mathbb{Z}^{n-2}\\ F_{\bfz}(X_{1},X_{2})\neq\square}}\sum_{\substack{(x_{1},x_{2})\in\mathcal{B}\cap\mathbb{Z}^{2}\\F_{\bfz}(x_{1},x_{2})=\square}}1&\ll_{n,\varepsilon} \#(\mathcal{C}\cap\mathbb{Z}^{n-2})\cdot D_{K}^{\frac{4}{3n}-\frac{2\delta_{2,K}}{3}+\varepsilon}\\&\ll D_{K}^{\sum_{j=3}^{n}\left(\frac{1}{n}-\delta_{j,K}\right)+\frac{4}{3n}-\frac{2\delta_{2,K}}{3}+\varepsilon}\\&\ll_{n} D_{K}^{\frac{1}{2}-\frac{2}{3n}+\frac{\delta_{2,K}}{3}+\varepsilon}.
\end{split}
\]
On the other hand, thanks to Corollary $\ref{cor : square}$ we know that
\[
 \sum_{\substack{\bfz\in\mathcal{C}\cap Z_{\square}}}\sum_{\substack{(x_{1},x_{2})\in\mathcal{B}\cap\mathbb{Z}^{2}\\F_{\bfz}(x_{1},x_{2})=\square}}1\ll_{n,\varepsilon} D_{K}^{\frac{1}{4}+\frac{1}{n}+\varepsilon},
\]
when $n\geq 8$, and the contribution is zero otherwise. Since $\frac14+\frac1n\leq\frac12-\frac{2}{3n}$ for $n\geq 8$, this completes the proof of Theorem \ref{thm : deltalarge}.
\subsection{The square sieve for composite moduli: bounding $h_{2}(K)$ when $\delta_{2,K}$ is large.}
In this section, we are going to prove Theorem $\ref{thm : deltasmall}$. As mentioned in the introduction, we do this by deploying the square sieve for composite moduli. We record here the sieve Lemma as stated in \cite[Lemma $1$]{HP12}. In what follows $\mathcal{U}$, and $\mathcal{V}$ are disjoint sets of primes and $\mathcal{A}=\{uv:u\in\mathcal{U},v\in\mathcal{V}\}$
\begin{lem}
Let $A = |\mathcal{A}|$, $U = |\mathcal{U}|$ , and $V = |\mathcal{V}|$, so that $A = U V$ . Furthermore, assume that $V^{3}\ll A$. Let $\omega$ be a non-negative weight such that $\omega (m) = 0$ for $m=0$ or $|m|>\exp \min ( U,V)$, then
\[
\begin{split}
\sum_{k}\omega (k^{2})&\ll \frac{1}{A}\sum_{m}\omega (m)+\frac{1}{A^{2}}\sum_{v,v'\in\mathcal{V}}\sum_{\substack{u,u'\in\mathcal{U}\\u\neq u'}}\left|\sum_{m\in\mathbb{Z}}\left(\frac{m}{vv'}\right)\left(\frac{m}{uu'}\right)\omega(m)\right|\\&\quad+\frac{U}{A^{2}}\sum_{\substack{v,v'\in\mathcal{V}\\v\neq v'}}\left|\sum_{m\in\mathbb{Z}}\left(\frac{m}{vv'}\right)\omega(m)\right|.
\end{split}
\]
\label{Lem : sqsievecom}
\end{lem}
In the remainder of this section we will use Lemma $\ref{Lem : sqsievecom}$, to prove the following
\begin{lem}
    Let $\bfz\in\mathbb{Z}^{n-2}\cap\mathcal{C}$ satisfy the following conditions:
    \begin{itemize}
        \item[$i)$] The polynomial $F_{\bfz}(X_{1},X_{2})$ is not a square in $\mathbb{Z}[X_{1},X_{2}]$.
        \item[$ii)$] For every $\sigma,\tau\in \Emb (K/\mathbb{Q})\setminus\{\Id\}$, with $\sigma\neq\tau$ and $\tau(\alpha_{2,K})\neq\alpha_{2,K}$,
        \[
        (\tau(\alpha_{2,K})-\alpha_{2,K})(\sigma(z)-z)\neq(\sigma(\alpha_{2,K})-\alpha_{2,K})(\tau(z)-z),
        \]
        \item[$iii)$] For every $\sigma\in \Emb(K/\mathbb{Q})\setminus\{\Id\}$, $z\not\in K^{\sigma}$.
        \item[$iv)$] For every $y\in \overline{\mathbb{Q}}$, the polynomial $F_{\bfz}(X,y)$ is not a square in $\overline{\mathbb{Q}}[X]$.
    \end{itemize}
   
    Let $(\omega_{\bfz}(m))_{m}$ be the sequence defined in $(\ref{eq : weight})$. Then,
    \[
    \sum_{k}\omega_{\bfz}(k^{2})\ll_{n,\varepsilon} D_{K}^{\frac{2}{n}-\delta_{2,K}-\frac{1}{2n}-\left(\frac{1}{7n}-\frac{3\delta_{2,K}}{14}\right)+\varepsilon},
    \]
    if $\delta_{2,K}\geq \frac{8}{21n}$.
    \label{lem : deltasmall}
\end{lem}
 We will then use Lemma \ref{lem : deltasmall} to prove
Theorem \ref{thm : deltasmall}.\newline
To prove Lemma $\ref{lem : deltasmall}$, we start by applying Lemma $\ref{Lem : sqsievecom}$, getting
\begin{equation}
\begin{split}
\sum_{k\in\mathbb{Z}}\omega_{\bfz}(k^{2})&\ll\frac{1}{A}\sum_{m}\omega_{\bfz}(m)+\frac{1}{A^{2}}\sum_{v,v'\in\mathcal{V}}\sum_{\substack{u,u'\in\mathcal{U}\\u\neq u'}}\left|\sum_{m\in\mathbb{Z}}\left(\frac{m}{uu'}\right)\left(\frac{m}{vv'}\right)\omega_{\bfz}(m)\right|\\&\quad+\frac{U}{A^{2}}\sum_{\substack{v,v'\in\mathcal{V}\\v\neq v'}}\left|\sum_{m\in\mathbb{Z}}\left(\frac{m}{vv'}\right)\omega_{\bfz}(m)\right|.
\end{split}
\label{eq : sievecomp}
\end{equation}
where $\mathcal{U},\mathcal{V}$ are disjoint finite families of primes that will be chosen later. Again, the first term on the right-hand side of $(\ref{eq : sievecomp})$ can be easily bounded by
\begin{equation}
\frac{1}{A}\sum_{m}\omega_{\bfz}(m)\ll \frac{D_{K}^{\frac{2}{n}-\delta_{2,K}}}{A}.
\label{eq : easypartcomp}
\end{equation}
Hence, we need to focus on the two other terms. We will refer to the term
\begin{equation}
M=\frac{1}{A^{2}}\sum_{v,v'\in\mathcal{V}}\sum_{\substack{u,u'\in\mathcal{U}\\u\neq u'}}\left|\sum_{m\in\mathbb{Z}}\left(\frac{m}{uu'}\right)\left(\frac{m}{vv'}\right)\omega_{\bfz}(m)\right|,
\label{eq : mainsieve}
\end{equation}
as the main sieve and to the term
\[
N=\frac{U}{A^{2}}\sum_{\substack{v,v'\in\mathcal{V}\\v\neq v'}}\left|\sum_{m\in\mathbb{Z}}\left(\frac{m}{vv'}\right)\omega_{\bfz}(m)\right|,
\]
as the prime sieve.
\subsection*{Proof of Lemma $\ref{lem : deltasmall}$: the main sieve, preliminaries.}
In what follows we denote $q_{1}=uu'$, $q_{2}=vv'$.
We start again from
\[
\begin{split}
S(\bfz;q_{1}q_{2})&=\sum_{(x_{1},x_{2})\in\mathbb{Z}^{2}}g_{1}(x_{1})g_{2}(x_{2})
\left(\frac{F_{\bfz}(x_{1},x_{2})}{q_{1}q_{2}}\right)
\end{split}
\]
but this time, the first step will be to extend only the summation in $x_{1}$ over all the classes modulo $q_{1}q_{2}$:
\[
\begin{split}
S(\bfz;q_{1}q_{2})&=\sum_{x_{2}\in\mathbb{Z}}g_{2}(x_{2})\sum_{y_{1}\text{ }(q_{1}q_{2}) }\left(\frac{F_{\bfz}(y_{1},x_{2})}{q_{1}q_{2}}\right)\sum_{\substack{x_{1}\in\mathbb{Z}\\x_{1}\equiv y_{1}\text{ }(q_{1}q_{2})}}g_{1}(x_{1})\\&=\frac{1}{q_{1}q_{2}}\sum_{u_{1}}\hat{g}_{1}\left(\frac{u_{1}}{q_{1}q_{2}}\right)S(\bfz,u_{1};q_{1}q_{2}),
\end{split}
\]
where 
\[
S(\bfz,u_{1};q_{1}q_{2})=\sum_{x_{2}\in\mathbb{Z}}g_{2}(x_{2})Z(u_{1},x_{2};q_{1}q_{2}),
\]
and where for every $r$, 
\[
Z(u_{1},x_{2};r)=\sum_{y_{1} \text{ }(r)}\left(\frac{F_{\bfz}(y_{1},x_{2})}{r}\right)e_{r}(u_{1}y_{1}).
\]
Applying the Chinese remainder theorem again, we obtain
\[
S(\bfz,u_{1};q_{1}q_{2})=\sum_{x_{2}\in\mathbb{Z}}g_{2}(x_{2})Z(u_{1}\overline{q_{2}},x_{2};q_{1})Z(u_{1}\overline{q_{1}},x_{2};q_{2}).
\]
We now apply the $q$-analogue of van der Corput's method. The idea is to shift the sum over $x_{2}$ by multiples of $q_{2}$, leaving the factor modulo $q_{2}$ unchanged. Combining these shifts with the Cauchy--Schwarz inequality reduces the problem to estimating correlation sums involving the factor with the smaller modulus $q_{1}$. For an extensive discussion of the $q$-analogue of van der Corput's method, we refer the reader to \cite{WX21}. Set
\[
H=\left\lfloor
\frac{D_K^{\frac1n-\delta_{2,K}}}{q_{2}}
\right\rfloor+1.
\]
Then
\[
\begin{split}
HS(\bfz,u_{1};q_{1}q_{2})&=\sum_{0\leq h < H}\sum_{x_{2}\in\mathbb{Z}}g_{2}(x_{2}+q_{2}h)Z(u_{1}\overline{q_{2}},x_{2}+q_{2}h;q_{1})Z(u_{1}\overline{q_{1}},x_{2}+q_{2}h;q_{2})\\&=\sum_{x_{2}\in I}Z(u_{1}\overline{q_{1}},x_{2};q_{2})\sum_{0\leq h < H}g_{2}(x_{2}+q_{2}h)Z(u_{1}\overline{q_{2}},x_{2}+q_{2}h;q_{1}),
\end{split}
\]
where $I= [-2D_{K}^{\frac{1}{n}-\delta_{2,K}}-q_{2}(H-1),2D_{K}^{\frac{1}{n}-\delta_{2,K}}]$. Applying Cauchy–Schwarz,
\[
H^{2}|S(\bfz,u_{1};q_{1}q_{2})|^{2}\ll \Sigma_{1}\Sigma_{2},
\]
where
\[
\begin{split}
&\Sigma_{1}=\sum_{x_{2}\in I}|Z(u_{1}\overline{q_{1}},x_{2};q_{2})|^{2}\\&\Sigma_{2}=\sum_{x_{2}\in I}\left|\sum_{0\leq h < H}g_{2}(x_{2}+q_{2}h)Z(u_{1}\overline{q_{2}},x_{2}+q_{2}h;q_{1})\right|^{2}.
\end{split}
\]
Expanding the square in $\Sigma_{2}$, we get
\[
\Sigma_{2}=\sum_{0\leq h_{1},h_{2}< H}R(h_{1},h_{2}),
\]
where
\[
R(h_{1},h_{2})=\sum_{\substack{x_{2}\in\mathbb{Z}}}g_{2}(x_{2}+q_{2}h_{1})g_{2}(x_{2}+q_{2}h_{2})Z(u_{1}\overline{q_{2}},x_{2}+q_{2}h_{1};q_{1})\overline{Z(u_{1}\overline{q_{2}},x_{2}+q_{2}h_{2};q_{1})}.
\]
Since $R(h_{1},h_{2})=R(h_{1}-h_{2},0)$, we get that
\[
\begin{split}
\Sigma_{2}&=\sum_{0\leq h_{1},h_{2}< H}R(h_{1}-h_{2},0)\\&=\sum_{-H<h< H}(H-|h|)R(h,0)\\&\ll H\sum_{-H<h<H}\left|\sum_{\substack{x_{2}\in\mathbb{Z}}}g_{2}(x_{2})g_{2}(x_{2}+q_{2}h)Z(u_{1}\overline{q_{2}},x_{2};q_{1})\overline{Z(u_{1}\overline{q_{2}},x_{2}+q_{2}h;q_{1})}\right|,
\end{split}
\]
Then
\[
\Sigma_{2}\leq H(\Sigma_{2A}+\Sigma_{2B}),
\]
where,
\[
\begin{split}
&\Sigma_{2A}=\sum_{\substack{x_{2}\in\mathbb{Z}}}g_{2}(x_{2})^{2}|Z(u_{1}\overline{q_{2}},x_{2};q_{1})|^{2}\\&\Sigma_{2B}=\sum_{\substack{-H< h< H\\h\neq 0}}\left|\sum_{\substack{x_{2}\in\mathbb{Z}}}g_{2}(x_{2})g_{2}(x_{2}+q_{2}h)Z(u_{1}\overline{q_{2}},x_{2};q_{1})\overline{Z(u_{1}\overline{q_{2}},x_{2}+q_{2}h;q_{1})}\right|.
\end{split}
\]
Next we are going to deal with $\Sigma_{1},\Sigma_{2A}$, and $\Sigma_{2B}$ separately:
\subsubsection{Bounding $\Sigma_{1}$, and $\Sigma_{2A}$} We need to bound the sum
\[
\begin{split}
&\Sigma_{1}=\sum_{x_{2}\in I}|Z(u_{1}\overline{q_{1}},x_{2};q_{2})|^{2}.
\end{split}
\]
where for every integer $r$
\[
Z(s,t;r)=\sum_{y_{1} \text{ }(r)}\left(\frac{F_{\bfz}(y_{1},t)}{r}\right)e_{r}(sy_{1}).
\]
Since $r=q_{2}=vv'$ with $v$ and $v'$ primes, we distinguish two cases
\begin{itemize}
    \item[$i)$] the case $v\neq v'$. In this case we can use Lemma $\ref{lem : exp1dim}$,
    \[
    Z(s,t;q_{2})=O(\sqrt{q_{2}}),
    \]
    for every $s$, provided that
    \[
    \mathcal{V}\subset\{p:\forall y\in\overline{\mathbb{F}}_{p},\quad F_{\bfz}(X,y)\text{ is not a square in } \overline{\mathbb{F}}_{p}[X]\}.
    \]
    Hence, one gets
    \[
    \Sigma_{1}\ll D_{K}^{\frac{1}{n}-\delta_{2,K}}q_{2}.
    \]
    \item[$ii)$] if $v=v'$ we use Lemma $\ref{lem : exp_{p}2}$ getting
    \[
    \Sigma_{1}\ll D_{K}^{\frac{1}{n}-\delta_{2,K}}((q_{2}+\delta_{q_{2}|u_{1}}(1-\mu(q_{2}))q_{2}^{2})),
    \]
    where $\delta_{q_{2}|u_{1}}=1$ if $q_{2}|u_{1}$ and zero otherwise, and $\mu$ is the M\"obius function.
\end{itemize}
Similarly one proves that
\[
\Sigma_{2A}\ll_{n} D_{K}^{\frac{1}{n}-\delta_{2,K}}q_{1},
\]
provided 
\[
    \mathcal{U}\subset\{p:\forall y\in\overline{\mathbb{F}}_{p},\quad F_{\bfz}(X,y)\text{ is not a square in } \overline{\mathbb{F}}_{p}[X]\}.
    \]

Thus, it remains to estimate $\Sigma_{2B}$. To do this, we first denote $g_{h}(x_{2})=g_{2}(x_{2})g_{2}(x_{2}+q_{2}h)$. Applying the Poisson summation formula, one gets
\[
\Sigma_{2B}\ll\frac{1}{q_{1}}\sum_{\substack{-H<h<H\\h\neq 0}}\left|\sum_{u_{2}\in\mathbb{Z}}\hat{g}_{h}\left(\frac{u_{2}}{q_{1}}\right)W(u_{1}\overline{q_{2}},u_{2},q_{2}h;q_{1})\right|
\]
where 
\[
W(s,t,r;q_{1})=\sum_{x,y,w\text{ }(q_{1})}\left(\frac{F_{\bfz}(x,y)}{q_{1}}\right)\left(\frac{F_{\bfz}(w,y+r)}{q_{1}}\right)e_{q_{1}}(s(x-w)+ty)
\]
Applying Proposition $\ref{prop : Anumber}$, we have
\begin{cor}
Let $K/\mathbb{Q}$ be a finite extension of degree $n$ and let $L$ be its Galois closure. Let $\bfz=(z_{3},...,z_{n})$, and $q_{1}=uu'$ with $u,u'$ distinct odd primes such that:
\begin{itemize}
    \item[$i)$] $u,u'\nmid \disc f_{\alpha_{2,K}}\disc f_{z}$, where $f_{\alpha_{2,K}}$ and $f_{z}$ are the minimal polynomials of $\alpha_{2,K}$ and $z$ over $\mathbb{Q}$,
    \item[$ii)$] for all $\sigma,\tau\in\Emb(K/\mathbb{Q})\setminus\{\Id\}$, with $\sigma\neq \tau$,  and $\tau(\alpha_{2,K})\neq \alpha_{2,K}$, then
    \[
u,u'\nmid N_{L}\left((\tau(\alpha_{2,K})-\alpha_{2,K}) (\sigma(z)-z)- (\sigma(\alpha_{2,K})-\alpha_{2,K})(\tau(z)-z)\right)
\]
\item[$iii)$] For every $y\in\overline{\mathbb{F}}_{u}$ (resp. in $\overline{\mathbb{F}}_{u'}$), the polynomial $F_{\bfz}(X,y)$ is not a square over $\overline{\mathbb{F}}_{u}[X]$ (resp. in $\overline{\mathbb{F}}_{u'}[X]$)
\end{itemize}
Then, for any $s\in\mathbb{Z}/q_{1}\mathbb{Z}$, there exists $A_{s,u}\subset\{(r,t)\in\left(\mathbb{Z}/u\mathbb{Z}\right)\times \left(\mathbb{Z}/u\mathbb{Z}\right)\}$, and $A_{s,u'}\subset\{(r,t)\in\left(\mathbb{Z}/u'\mathbb{Z}\right)\times \left(\mathbb{Z}/u'\mathbb{Z}\right)\}$,  such that
\begin{itemize}
\item[$a)$] for any $(r,t)$ such that $(r,t)\mod u\not\in A_{s,u}$, and $ (r,t)\mod u'\not\in A_{s,u'}$,  one has 
\[
\left|\sum_{x,y,w \text{ }(q_{1})}\left(\frac{F_{\bfz}(x,y)}{q_{1}}\right)\left(\frac{F_{\bfz}(w,y+r)}{q_{1}}\right)e_{q_{1}}(s(x-w)+ty)\right|\ll_{n} q_{1}^{3/2}(s,t,q_{1})^{\frac{1}{2}}.
\]
\item[$b)$] $\# A_{s,u}=O_{n}(1)$, and $\# A_{s,u'}=O_{n}(1)$.
\end{itemize}
\label{cor : serve}
\end{cor}
We are ready to bound the contribution of $\Sigma_{2B}$. Let $u,u'$ be two primes which satisfy the hypotheses of Corollary $\ref{cor : serve}$.
Moreover, we assume that $\frac{q_{1}^{\frac{1}{2}}}{2}\leq u,u'\leq 2q_{1}^{\frac{1}{2}}$. Since the bound for $W(u_{1}\overline{q_{2}},u_{2},q_{2}h;q_{1})$ in Corollary \ref{cor : serve} depends on the greatest common divisor of $q_{1}$, $u_{1}\overline{q_{2}}$, and $u_{2}$, we have to handle two cases.
\begin{itemize}
    \item[$i)$] the case when $(u_{1},q_{1})=1$. We start by applying Corollary $\ref{cor : serve}$, getting
\[
\begin{split}
\Sigma_{2B}&\ll\frac{1}{q_{1}}\sum_{\substack{-H<h<H\\h\neq 0}}\sum_{\substack{u_{2}\in\mathbb{Z}\\(q_{2}h,u_{2})\not\in A_{u_{1}\overline{q_{2}},u} \\(q_{2}h,u_{2})\not\in A_{u_{1}\overline{q_{2}},u'}}}\left|\hat{g}_{h}\left(\frac{u_{2}}{q_{1}}\right)\right||W(u_{1}\overline{q_{2}},u_{2},q_{2}h;q_{1})|\\&\quad+ \frac{1}{q_{1}}\sum_{\substack{-H<h<H\\h\neq 0}}\sum_{\substack{u_{2}\in\mathbb{Z}\\(q_{2}h,u_{2})\in A_{u_{1}\overline{q_{2}},u} }}\left|\hat{g}_{h}\left(\frac{u_{2}}{q_{1}}\right)\right||W(u_{1}\overline{q_{2}},u_{2},q_{2}h;q_{1})|\\&\quad+ \frac{1}{q_{1}}\sum_{\substack{-H<h<H\\ h\neq 0}}\sum_{\substack{u_{2}\in\mathbb{Z}\\(q_{2}h,u_{2})\in A_{u_{1}\overline{q_{2}},u'} }}\left|\hat{g}_{h}\left(\frac{u_{2}}{q_{1}}\right)\right||W(u_{1}\overline{q_{2}},u_{2},q_{2}h;q_{1})|\\&\ll_{n} Hq_{1}^{\frac{3}{2}}+ q_{1}^{2}\cdot\frac{1}{q_{1}}\sum_{\substack{-H<h<H\\ h\neq 0}}\sum_{\substack{u_{2}\in\mathbb{Z}\\(q_{2}h,u_{2})\in A_{u_{1}\overline{q_{2}},u} }}\left|\hat{g}_{h}\left(\frac{u_{2}}{q_{1}}\right)\right|\\&\quad + q_{1}^{2}\cdot\frac{1}{q_{1}}\sum_{\substack{-H<h<H\\ h\neq 0}}\sum_{\substack{u_{2}\in\mathbb{Z}\\(q_{2}h,u_{2})\in A_{u_{1}\overline{q_{2}},u'} }}\left|\hat{g}_{h}\left(\frac{u_{2}}{q_{1}}\right)\right|,
\end{split}
\]
provided  that we have $q_{1}\gg D_{K}^{\frac{1}{n}-\delta_{2,K}}$; now we have
\[
\begin{split}
\sum_{\substack{-H<h<H\\ h\neq 0}}\sum_{\substack{u_{2}\in\mathbb{Z}\\(q_{2}h,u_{2})\in A_{u_{1}\overline{q_{2}},u}}}\left|\hat{g}_{h}\left(\frac{u_{2}}{q_{1}}\right)\right|&\leq\sum_{(r,t)\in A_{u_{1}\overline{q_{2}},u}}\sum_{\substack{h\in[-H,H]\\h\neq 0\\q_{2}h=r\text{ }(u) }}\sum_{\substack{u_{2}\in\mathbb{Z}\\u_{2}=t\text{  }(u)}}\left|\hat{g}_{h}\left(\frac{u_{2}}{q_{1}}\right)\right|\\&\ll D_{K}^{\frac{1}{n}-\delta_{2,K}}\sum_{(r,t)\in A_{u_{1}\overline{q_{2}},u}}\sum_{\substack{h\in[-H,H]\\ q_{2}h=r\text{ }(u)}}\sum_{\substack{u_{2}\in\mathbb{Z}\\u_{2}=t\text{  }(u)}}\left|1+\frac{|u_{2}|D_{K}^{\frac{1}{n}-\delta_{2,K}}}{q_{1}}\right|^{-2}\\&\ll_{n} D_{K}^{\frac{1}{n}-\delta_{2,K}}\left(1+\frac{H}{q_{1}^{\frac{1}{2}}}\right),
\end{split}
\]
provided that $\frac{q_{1}}{u}\ll D_{K}^{\frac{1}{n}-\delta_{2,K}}$, i.e. $q_{1}^{\frac{1}{2}}\ll D_{K}^{\frac{1}{n}-\delta_{2,K}}$. Moreover, recalling that $H=\lfloor\frac{D_{K}^{\frac{1}{n}-\delta_{2,K}}}{q_{2}}\rfloor +1$, we conclude that
\[
\begin{split}
q_{1}^{2}\cdot\frac{1}{q_{1}}\sum_{\substack{-H<h<H}}\sum_{\substack{u_{2}\in\mathbb{Z}\\(q_{2}h,u_{2})\in A_{u_{1}\overline{q_{2}},u}} }\left|\hat{g}_{h}\left(\frac{u_{2}}{q_{1}}\right)\right|&\ll_{n} D_{K}^{\frac{1}{n}-\delta_{2,K}}\left(1+\frac{H}{q_{1}^{\frac{1}{2}}}\right)q_{1}\\&\ll Hq_{2}q_{1}+Hq_{1}^{\frac{3}{2}}\\&\ll Hq_{1}^{\frac{3}{2}},
\end{split}
\]
provided that $ \max\{q_{2}^{2},D_{K}^{\frac{1}{n}-\delta_{2,K}}\}\ll q_{1}\ll D_{K}^{\frac{2}{n}-2\delta_{2,K}}$. Hence, in this case, we have
\[
\Sigma_{2B}\ll_{n} Hq_{1}^{\frac{3}{2}},
\]
provided that $ \max\{q_{2}^{2},D_{K}^{\frac{1}{n}-\delta_{2,K}}\}\ll q_{1}\ll D_{K}^{\frac{2}{n}-2\delta_{2,K}}$.
\item[$ii)$] The case $(u_{1},q_{1})>1$. We start by writing
\begin{equation}
\begin{split}
\Sigma_{2B}&\ll\frac{1}{q_{1}}\sum_{\substack{-H<h<H\\h\neq 0}}\left|\sum_{\substack{u_{2}\in\mathbb{Z}\\(u_{2},q_{1})=1 }}\hat{g}_{h}\left(\frac{u_{2}}{q_{1}}\right)W(u_{1}\overline{q_{2}},u_{2},q_{2}h;q_{1})\right|\\&\quad + \frac{1}{q_{1}}\sum_{\substack{-H<h<H\\ h\neq 0}}\left|\sum_{\substack{u_{2}\in\mathbb{Z}\\(u_{2},q_{1})>1 }}\hat{g}_{h}\left(\frac{u_{2}}{q_{1}}\right)W(u_{1}\overline{q_{2}},u_{2},q_{2}h;q_{1})\right|
\end{split}
\label{eq : notcopr}
\end{equation}
observe that the contribution of the $u_{2}$ coprime to $q_{1}$ can be bounded as in $(i)$, getting
\[
\frac{1}{q_{1}}\sum_{\substack{-H<h<H\\h\neq 0}}\left|\sum_{\substack{u_{2}\in\mathbb{Z}\\(u_{2},q_{1})=1 }}\hat{g}_{h}\left(\frac{u_{2}}{q_{1}}\right)W(u_{1}\overline{q_{2}},u_{2},q_{2}h;q_{1})\right|\ll_{n} Hq_{1}^{\frac{3}{2}},
\]
provided that $ \max\{q_{2}^{2},D_{K}^{\frac{1}{n}-\delta_{2,K}}\}\ll q_{1}\ll D_{K}^{\frac{2}{n}-2\delta_{2,K}}$. For the contribution of the second sum appearing on the right-hand side of $(\ref{eq : notcopr})$, we are going to use the trivial bound $\ll q_{1}^{2}$ for the exponential sum (see Lemma $\ref{lem : exptrivial}$), getting
\[
\begin{split}
\frac{1}{q_{1}}&\sum_{\substack{-H<h<H\\ h\neq 0}}\left|\sum_{\substack{u_{2}\in\mathbb{Z}\\(u_{2},q_{1})>1 }}\hat{g}_{h}\left(\frac{u_{2}}{q_{1}}\right)W(u_{1}\overline{q_{2}},u_{2},q_{2}h;q_{1})\right|\\&\ll_{n} q_{1}\sum_{\substack{-H<h<H\\ h\neq 0}}\sum_{\substack{u_{2}\in\mathbb{Z}\\(u_{2},q_{1})>1 }}\left|\hat{g}_{h}\left(\frac{u_{2}}{q_{1}}\right)\right|\\&\ll H\cdot q_{1}\cdot D_{K}^{\frac{1}{n}-\delta_{2,K}},
\end{split}
\]
provided that $q_{1}^{\frac{1}{2}}\ll D_{K}^{\frac{1}{n}-\delta_{2,K}}$. Thus
\[
\begin{split}
\frac{1}{q_{1}}\sum_{\substack{-H<h<H}}\left|\sum_{\substack{u_{2}\in\mathbb{Z}\\(u_{2},q_{1})>1 }}\hat{g}_{h}\left(\frac{u_{2}}{q_{1}}\right)W(u_{1}\overline{q_{2}},u_{2},q_{2}h;q_{1})\right|&\ll_{n} HD_{K}^{\frac{1}{n}-\delta_{2,K}}q_{1},
\end{split}
\]
provided that
 $\max\{q_{2}^{2},D_{K}^{\frac{1}{n}-\delta_{2,K}}\}\ll q_{1}\ll D_{K}^{\frac{2}{n}-2\delta_{2,K}}$.
\end{itemize}
Combining the estimates for $\Sigma_{1},\Sigma_{2A}$, and $\Sigma_{2B}$, we conclude
\begin{lem}
 Let $\bfz\in\mathbb{Z}^{n-2}\setminus Z_{F}$. Let $u,u',v,v'$ be odd prime numbers with $u\neq u'$ and $(uu',vv')=1$, such that
    \begin{itemize}
    \item[$i)$]
    $u,u'\nmid \disc(f_{\alpha_{2,K}})\disc(f_z)$.

    \item[$ii)$]
    For every $y\in\overline{\mathbb{F}}_u$ (resp. $\overline{\mathbb{F}}_{u'},\overline{\mathbb{F}}_v$, and $\overline{\mathbb{F}}_{v'}$), the polynomial
    $F_{\bfz}(X,y)$ is not a square in
    $\overline{\mathbb{F}}_u[X]$ (resp. $\overline{\mathbb{F}}_{u'}[X],\overline{\mathbb{F}}_v[X]$, and $\overline{\mathbb{F}}_{v'}[X]$).

    \item[$iii)$]
    For every $\sigma,\tau\in\Emb(K/\mathbb{Q})
    \setminus\{\Id\}$ with $\sigma\neq\tau$ and
    $\tau(\alpha_{2,K})\neq\alpha_{2,K}$, one has
    \[
    u,u'\nmid N_L\left(
        (\tau(\alpha_{2,K})-\alpha_{2,K})(\sigma(z)-z)
        -(\sigma(\alpha_{2,K})-\alpha_{2,K})
        \bigl(\tau(z)-z\bigr)
    \right).
    \]
\end{itemize}
Moreover, denote $q_{1}=uu'$, and $q_{2}=vv'$ and assume
\[
\frac{q_{1}^{1/2}}{2}\leq u,u'\leq 2q_{1}^{1/2}.
\]
Then, one has
\[
S(\bfz,u_{1};q_{1}q_{2})\ll_{n}
\kappa_{q_{2}}(u_{1})\begin{cases}
D_{K}^{\frac{1}{2n}-\frac{\delta_{2,K}}{2}}q_{1}    &\text{ if }(u_{1},q_{1})=1\\
\frac{D_{K}^{\frac{1}{n}-\delta_{2,K}}q_{1}}{q_{2}^{\frac{1}{2}}}   &\text{otherwise.}

\end{cases}
\]
provided that $ \max\{q_{2}^{2},D_{K}^{\frac{1}{n}-\delta_{2,K}}\}\ll  q_{1}\ll D_{K}^{\frac{2}{n}-2\delta_{2,K}}$, and where
\[
\kappa_{q_{2}}(u_{1})=
\begin{cases}
q_{2}^{\frac{1}{2}}     &\text{if $q_{2}=v^{2}$, and $q_{2}|u_{1}$}\\
1                       &\text{otherwise.}
\end{cases}
\]
\label{lem : mainsieve}
\end{lem}
Then we can finally bound the main sieve. The first step is to choose the families of primes $\mathcal{U}$, and $\mathcal{V}$. Let $\mathcal{U}$ be the set of primes $u$ with
$Q_{\mathcal{U}}\leq u\leq 2Q_{\mathcal{U}}$
satisfying the following conditions:
\begin{itemize}
    \item[$i)$]
    $u\nmid \disc(f_{\alpha_{2,K}})\disc(f_z)$.

    \item[$ii)$]
    For every $y\in\overline{\mathbb{F}}_u$, the polynomial
    $F_{\bfz}(X,y)$ is not a square in
    $\overline{\mathbb{F}}_u[X]$.

    \item[$iii)$]
    For every $\sigma,\tau\in\Emb(K/\mathbb{Q})
    \setminus\{\Id\}$ with $\sigma\neq\tau$ and
    $\tau(\alpha_{2,K})\neq\alpha_{2,K}$, one has
    \[
    u\nmid N_L\left(
        (\tau(\alpha_{2,K})-\alpha_{2,K})(\sigma(z)-z)
        -(\sigma(\alpha_{2,K})-\alpha_{2,K})(\tau(z)-z)\right).
    \]
\end{itemize}
The other family is defined instead as follows:
\[
\begin{split}
    \mathcal{V}=\left\{v:\begin{matrix}Q_{\mathcal{V}}\leq v\leq 2Q_{\mathcal{V}}\\\forall y\in\overline{\mathbb{F}}_{v},\quad F_{\bfz}(X,y)\text{ is not a square in } \overline{\mathbb{F}}_{v}[X]\end{matrix}\right\},
\end{split}
\]
Moreover, we require
$D_{K}^{\frac{1}{2n}-\frac{\delta_{2,K}}{2}}\leq Q_{\mathcal{U}}\leq D_{K}^{\frac{1}{n}-\delta_{2,K}}$ and $Q_{\mathcal{U}}=Q_{\mathcal{V}}^{2}$, and $Q_{\mathcal{U}}Q_{\mathcal{V}}^{2}\leq D_{K}^{\frac{1}{n}}\leq (Q_{\mathcal{U}}Q_{\mathcal{V}})^{2}$. Observe that
\[
U=\#\mathcal{U}\asymp \frac{Q_{\mathcal{U}}}{\log Q_{\mathcal{U}}},\quad V=\#\mathcal{V}\asymp \frac{Q_{\mathcal{V}}}{\log Q_{\mathcal{V}}}. \]
Indeed, the only delicate point is to check that the condition
\[
\text{for every }y\in\overline{\mathbb{F}}_{p},\quad
F_{\bfz}(X,y)\text{ is not a square in }
\overline{\mathbb{F}}_{p}[X]
\]
holds for all but $O_{n}(\log D_K)$ primes. If $n$ is odd, this
condition holds automatically, so we may assume that $n$ is even.
Let $L$ be the Galois closure of $K$. We have
\[
F_{\bfz}(X,Y)=\prod_{i=1}^{n}\bigl(X+\sigma_i(\alpha_{2,K})Y+\sigma_i(z)\bigr).
\]
After reduction modulo a prime ideal of $L$ above $p$, the
specialization $F_{\bfz}(X,y)$ is a square if and only if there
exists a fixed-point-free involution $g\in S_{n}$ such that
\[
L_{i,g}(y)=0\quad\text{for every }i=1,\ldots,n,
\]
where
\[
L_{i,g}(Y)=\bigl(\sigma_i(\alpha_{2,K})-\sigma_{g(i)}(\alpha_{2,K})\bigr)Y+\sigma_i(z)-\sigma_{g(i)}(z)\in\mathcal{O}_L[Y].
\]

By hypothesis, $F_{\bfz}(X,y)$ is not a square in
$\overline{\mathbb{Q}}[X]$ for any
$y\in\overline{\mathbb{Q}}$. Thus, for each such involution $g$,
the polynomials $L_{i,g}$ have no common root over
$\overline{\mathbb{Q}}$. Consequently, either some $L_{i,g}$
is a nonzero constant, or there exist $i,j$ such that
$L_{i,g}$ and $L_{j,g}$ both have degree one and
\[
\operatorname{Res}(L_{i,g},L_{j,g})\neq 0.
\]
For each $g$, let $R_{g}\in\mathcal{O}_L\setminus\{0\}$ be
such a nonzero constant in the first case, and such a
resultant in the second case.

If there exists $y\in\overline{\mathbb{F}}_{p}$ for which
$F_{\bfz}(X,y)$ is a square, the corresponding system
$L_{i,g}(y)=0$ has a solution after reduction modulo a prime ideal above $p$. It follows that
\[
p\mid N_{L/\mathbb{Q}}(R_{g})
\]
for some fixed-point-free involution $g$. Since $\|\alpha_{2,K}\|,\|z\|\ll_{n}D_{K}$ and $[L:\mathbb{Q}]\leq n!$, we have
\[
0<\left|\prod_{\substack{g\in S_{n}\\g^{2}=\mathrm{Id}\\g(i)\neq i\ \forall i}}N_{L/\mathbb{Q}}(R_g)\right|\ll_{n} D_K^{O_{n}(1)}.
\]
The number of prime divisors of this nonzero integer is $O_{n}(\log D_K)$, which proves the claim.\newline
We split
\[
S(\bfz;q_{1}q_{2})=\frac{1}{q_{1}q_{2}}\sum_{\substack{u_{1}\in\mathbb{Z}\\ (u_{1},q_{1})=1}}\hat{g}_{1}\left(\frac{u_{1}}{q_{1}q_{2}}\right)S(\bfz,u_{1};q_{1}q_{2})+\frac{1}{q_{1}q_{2}}\sum_{\substack{u_{1}\in\mathbb{Z}\\ (u_{1},q_{1})>1}}\hat{g}_{1}\left(\frac{u_{1}}{q_{1}q_{2}}\right)S(\bfz,u_{1};q_{1}q_{2})
\]
We start by dealing with the contribution of those $u_{1}$ such that $(u_{1},q_{1})>1$. An application of Lemma $\ref{lem : mainsieve}$ combined with our assumption that $Q_{\mathcal{U}}Q_{\mathcal{V}}^{2}\leq D_{K}^{\frac{1}{n}}$, gives
\[
\begin{split}
\frac{1}{q_{1}q_{2}}\sum_{\substack{u_{1}\in\mathbb{Z}\\ (u_{1},q_{1})>1}}\hat{g}_{1}\left(\frac{u_{1}}{q_{1}q_{2}}\right)S(\bfz,u_{1};q_{1}q_{2})&\ll_{n} \frac{\kappa_{q_{2}}D_{K}^{\frac{1}{n}-\delta_{2,K}}}{q_{2}^{\frac{3}{2}}}\sum_{\substack{u_{1}\in\mathbb{Z}\\ (u_{1},q_{1})>1}}\left|\hat{g}_{1}\left(\frac{u_{1}}{q_{1}q_{2}}\right)\right|\\& \ll\frac{\kappa_{q_{2}}D_{K}^{\frac{2}{n}-\delta_{2,K}}}{Q_{\mathcal{U}}Q_{\mathcal{V}}},
\end{split}
\]
where
\[
\kappa_{q_{2}}=
\begin{cases}
q_{2}^{\frac{1}{2}}     &\text{if $q_{2}=v^{2}$}\\
1                       &\text{otherwise.}
\end{cases}
\]
 Let us turn our attention to those $u_{1}$ such that $(u_{1},q_{1})=1$. Applying Lemma $\ref{lem : mainsieve}$, combined with our assumption that $(Q_{\mathcal{U}}Q_{\mathcal{V}})^{2}\geq D_{K}^{\frac{1}{n}}$, gives
\[
\begin{split}
\frac{1}{q_{1}q_{2}}\sum_{\substack{u_{1}\in\mathbb{Z}\\ (u_{1},q_{1})=1}}\hat{g}_{1}\left(\frac{u_{1}}{q_{1}q_{2}}\right)S(\bfz,u_{1};q_{1}q_{2})&\ll_{n} \kappa_{q_{2}}D_{K}^{\frac{1}{2n}-\frac{\delta_{2,K}}{2}}Q_{\mathcal{U}}^{2}
\end{split}
\]
Combining the two bounds, we get
\[
S(\bfz;q_{1}q_{2})\ll_{n} \kappa_{q_{2}}\left(D_{K}^{\frac{1}{2n}-\frac{\delta_{2,K}}{2}}Q_{\mathcal{U}}^{2}+\frac{D_{K}^{\frac{2}{n}-\delta_{2,K}}}{Q_{\mathcal{U}}Q_{\mathcal{V}}}\right)
\]
Inserting this bound into the main sieve $(\ref{eq : mainsieve})$, and using the fact that
\[
\sum_{v,v'\in\mathcal{V}}\kappa_{vv'}\ll_{\varepsilon} Q_{\mathcal{V}}^{2+\varepsilon},
\]
we get
\begin{equation}
\frac{1}{A^{2}}\sum_{v,v'\in\mathcal{V}}\sum_{\substack{u,u'\in\mathcal{U}\\u\neq u'}}\left|\sum_{m\in\mathbb{Z}}\left(\frac{m}{uu'}\right)\left(\frac{m}{vv'}\right)\omega_{\bfz}(m)\right|\ll_{n,\varepsilon}  D_{K}^{\frac{1}{2n}-\frac{\delta_{2,K}}{2}+\varepsilon}Q_{\mathcal{U}}^{2} +\frac{D_{K}^{\frac{2}{n}-\delta_{2,K}+\varepsilon}}{Q_{\mathcal{U}}Q_{\mathcal{V}}}.
\label{eq : boundmain}
\end{equation}
\subsubsection{Proof of Lemma $\ref{lem : deltasmall}$: the prime sieve}
We briefly discuss now how to handle the Prime sieve
\[
N=\frac{U}{A^{2}}\sum_{\substack{v,v'\in\mathcal{V}\\v\neq v'}}\left|\sum_{m\in\mathbb{Z}}\left(\frac{m}{vv'}\right)\omega_{\bfz}(m)\right|.
\]
Note that the inner sum in the right-hand side is the same as the one appearing in the second term of the right-hand side of $(\ref{eq : sqafter})$. Since we have
\[
\begin{split}
    \mathcal{V}=\left\{v:\begin{matrix}Q_{\mathcal{V}}\leq v\leq 2Q_{\mathcal{V}}\\\forall y\in\overline{\mathbb{F}}_{v},\quad F_{\bfz}(X,y)\text{ is not a square in } \overline{\mathbb{F}}_{v}[X] \end{matrix}\right\},
\end{split}
\]
arguing as in Corollary \ref{cor : sieve1}, and using the bound $T(\bfu,\bfz,q_{2})\ll q_{2}^{\frac{3}{2}}$ one gets
\begin{equation}
\left|\sum_{m\in\mathbb{Z}}\left(\frac{m}{vv'}\right)\omega_{\bfz}(m)\right|\ll_{n} \frac{D_{K}^{\frac{2}{n}-\delta_{2,K}}}{(vv')^{\frac{1}{2}}}\sum_{\substack{\bfu\in\mathbb{Z}^{2}}}\left(1+\frac{D_{K}^{\frac{1}{n}}|u_{1}|}{vv'}\right)^{-2}\cdot\left(1+\frac{D_{K}^{\frac{1}{n}-\delta_{2,K}}|u_{2}|}{vv'}\right)^{-2}.
\label{eq : primesieve}
\end{equation}
Since $v'v\leq 4Q_{\mathcal{V}}^{2}\ll D_{K}^{\frac{1}{n}-\delta_{2,K}}$, the inner sum in $(\ref{eq : primesieve})$ is $O(1)$. Thus
\[
\left|\sum_{m\in\mathbb{Z}}\left(\frac{m}{vv'}\right)\omega_{\bfz}(m)\right|\ll_{n} \frac{D_{K}^{\frac{2}{n}-\delta_{2,K}}}{Q_{\mathcal{V}}}.
\]
Hence,
\begin{equation}
\frac{U}{A^{2}}\sum_{\substack{v,v'\in\mathcal{V}\\v\neq v'}}\left|\sum_{m\in\mathbb{Z}}\left(\frac{m}{vv'}\right)\omega_{\bfz}(m)\right|\ll_{n,\varepsilon} \frac{D_{K}^{\frac{2}{n}-\delta_{2,K}+\varepsilon}}{Q_{\mathcal{U}}Q_{\mathcal{V}}},
\label{eq : finale}
\end{equation}
which is acceptable.
\subsubsection{Proof of Lemma \ref{lem : deltasmall}: conclusion}
Combining $(\ref{eq : sievecomp})$, $(\ref{eq : easypartcomp})$, $(\ref{eq : boundmain})$, and $(\ref{eq : finale})$, we obtain 
\[
\sum_{k}\omega_{\bfz} (k^{2})\ll_{n,\varepsilon} \frac{D_{K}^{\frac{2}{n}-\delta_{2,K}+\varepsilon}}{Q_{\mathcal{U}}Q_{\mathcal{V}}} + D_{K}^{\frac{1}{2n}-\frac{\delta_{2,K}}{2}+\varepsilon}Q_{\mathcal{U}}^{2}.
\]
Imposing $Q_{\mathcal{U}}=Q_{\mathcal{V}}^{2}$, the bound above is optimal as soon as
\[
\frac{D_{K}^{\frac{2}{n}-\delta_{2,K}}}{Q_{\mathcal{V}}^{3}} = D_{K}^{\frac{1}{2n}-\frac{\delta_{2,K}}{2}}Q_{\mathcal{V}}^{4}
\]
 which implies
 \[
 Q_{\mathcal{V}}=D_{K}^{\frac{1}{7n}+\frac{1}{14}\left(\frac{1}{n}-\delta_{2,K}\right)}.
 \]
 These choices satisfy the preceding size conditions, since $\delta_{2,K}\geq 8/(21n)>1/(5n)$ and $\delta_{2,K}\leq 1/(2(n-1))+O_{n}\left(\frac{1}{\log D_{K}}\right)<2/(3n)$ for large $D_K$. Moreover, $V^2\ll U$ and $\min(U,V)\gg_{n}\log D_K$, so all hypotheses of Lemma \ref{Lem : sqsievecom} hold. This concludes the proof of Lemma \ref{lem : deltasmall}.
\subsubsection{Proof of Theorem $\ref{thm : deltasmall}$}
We can now prove Theorem $\ref{thm : deltasmall}$. Again, thanks to Lemma $\ref{lem : index2}$, we may assume that $K$ does not contain a subfield of index $2$. We recall that 
\[\mathcal{C}=\left[-D_{K}^{\frac{1}{n}-\delta_{3,K}},D_{K}^{\frac{1}{n}-\delta_{3,K}}\right]\times\cdots\times\left[-D_{K}^{\frac{1}{n}-\delta_{n,K}},D_{K}^{\frac{1}{n}-\delta_{n,K}}\right],
\]
and that
\[
\begin{split}
   & Z_{\square}=\{\bfz\in\mathbb{Z}^{n-2}: F_{\bfz}(X_{1},X_{2})\text{ is a square}\};\\&  Z_{F}=\{\bfz\in\mathbb{Z}^{n-2}: F_{\bfz}(X_{1},X_{2})\text{ does not satisfy the hypotheses of Proposition $\ref{prop : Anumber}$}\}
\end{split}
\]
We start by writing
\begin{equation}
\begin{split}
h_{2}(K)&\ll\# \left\{(y,z_{1},...,z_{n})\in\mathbb{Z}^{n+1}:\begin{matrix}|z_{i}|\ll D_{K}^{\frac{1}{n}-\delta_{i,K}},\quad\text{for }i=1,...,n\\y^{2}=N_{K}(z_{1}+z_{2}\alpha_{2,K}+\cdots+z_{n}\alpha_{n,K})\end{matrix}     \right\}\\&\ll \sum_{\bfz\in\mathcal{C}\cap\mathbb{Z}^{n-2}}\sum_{\substack{(x_{1},x_{2})\in\mathcal{B}\cap\mathbb{Z}^{2}\\F_{\bfz}(x_{1},x_{2})=\square}}1\\&= \sum_{\substack{\bfz\in\mathcal{C}\cap Z_{\square}}}\sum_{\substack{(x_{1},x_{2})\in\mathcal{B}\cap\mathbb{Z}^{2}\\F_{\bfz}(x_{1},x_{2})=\square}}1+\sum_{\substack{\bfz\in\mathcal{C}\\ \bfz\in Z_{F}\setminus Z_{\square}}}\sum_{\substack{(x_{1},x_{2})\in\mathcal{B}\cap\mathbb{Z}^{2}\\F_{\bfz}(x_{1},x_{2})=\square}}1+\sum_{\substack{\bfz\in\mathcal{C}\cap\mathbb{Z}^{n-2}\\ \bfz\not\in Z_{\square}\cup Z_{F}}}\sum_{\substack{(x_{1},x_{2})\in\mathcal{B}\cap\mathbb{Z}^{2}\\F_{\bfz}(x_{1},x_{2})=\square}}1.
\end{split}
\label{eq : tog1}
\end{equation}
For the first two terms we have the bounds
\[
 \sum_{\substack{\bfz\in\mathcal{C}\cap Z_{\square}}}\sum_{\substack{(x_{1},x_{2})\in\mathcal{B}\cap\mathbb{Z}^{2}\\F_{\bfz}(x_{1},x_{2})=\square}}1=\begin{cases}
    0       &\text{if $n=5,6,7$}\\
    O_{n,\varepsilon}\left(D_{K}^{\frac{1}{4}+\frac{1}{n}+\varepsilon}\right) &\text{if $n\geq 8$}
\end{cases},\qquad\sum_{\substack{\bfz\in\mathcal{C}\\ \bfz\in Z_{F}\setminus Z_{\square}}}\sum_{\substack{(x_{1},x_{2})\in\mathcal{B}\cap\mathbb{Z}^{2}\\F_{\bfz}(x_{1},x_{2})=\square}}1\ll_{n,\varepsilon} D_{K}^{\frac{1}{2}-\frac{1}{2n}-\frac{1}{7n}+\varepsilon},
\]
thanks to Corollary $\ref{cor : square}$ and Corollary $\ref{cor : Ffail}$. For the third term in $(\ref{eq : tog1})$, we can use Lemma $\ref{lem : deltasmall}$, getting
\[
\begin{split}
\sum_{\substack{\bfz\in\mathcal{C}\cap\mathbb{Z}^{n-2}\\ \bfz\not\in Z_{\square}\cup Z_{F}}}\sum_{\substack{(x_{1},x_{2})\in\mathcal{B}\cap\mathbb{Z}^{2}\\F_{\bfz}(x_{1},x_{2})=\square}}1&\ll_{n,\varepsilon} \#(\mathcal{C}\cap\mathbb{Z}^{n-2})\cdot D_{K}^{\frac{2}{n}-\delta_{2,K}-\frac{1}{2n}-\left(\frac{1}{7n}-\frac{3\delta_{2,K}}{14}\right)+\varepsilon}\\&\ll D_{K}^{\sum_{j=3}^{n}\left(\frac{1}{n}-\delta_{j,K}\right)+\frac{2}{n}-\delta_{2,K}-\frac{1}{2n}-\left(\frac{1}{7n}-\frac{3\delta_{2,K}}{14}\right)+\varepsilon}\\&\ll_{n}  D_{K}^{\frac{1}{2}-\frac{1}{2n}-\left(\frac{1}{7n}-\frac{3\delta_{2,K}}{14}\right)+\varepsilon}.
\end{split}
\]
which completes the proof of Theorem $\ref{thm : deltasmall}$.
\section{The relevant exponential sums.}
\label{sec : exp}
This section is devoted to establishing all the exponential-sum estimates used in the proofs of Lemma $\ref{lem : exp}$ and Proposition $\ref{prop : Anumber}$. 
\subsection{$1$-dimensional exponential sum}
We begin by recording some standard estimates for exponential sums in one variable

\begin{lem}
Let $F(X,Y)\in\mathbb{F}_{p}[X,Y]$ be a square free polynomial, with $\deg_{X}F\geq 1$. Define 
\[
E_{F}=\{y\in \mathbb{F}_{p} : F(X,y) \text{ is a square in } \overline{\mathbb{F}}_{p}[X]\}. 
\]
Then for every $y\not\in E_{F}$,
\[
\left|\sum_{x\in\mathbb{F}_{p}}\left(\frac{F(x,y)}{p}\right)e_{p}(sx)\right|\ll_{\deg F}p^{\frac{1}{2}}.
\]
Moreover,  $\#E_{F}=O_{\deg F}(1)$.
\label{lem : exp1dim}
\end{lem}
\begin{proof}
   We may assume $p>\deg F$, otherwise the statement is trivial. If $s\neq 0$, then the result follows from \cite{We48}. If $s=0$, then from \cite[Theorem $11.23$]{IK04}, we have that 
    \[
\left|\sum_{x\in\mathbb{F}_{p}}\left(\frac{F(x,y)}{p}\right)\right|\ll_{\deg F}p^{\frac{1}{2}},
\]
for every $y\not\in E_{F}$. Hence, it is enough to show that $\#E_{F}=O_{\deg F}(1)$. Since $\deg_{X}F\geq 1$, $F(X,y)$ is constant only for $O_{\deg F}(1)$ values of $y$. It remains to show that $F(X,y)$ has a square factor only for $O_{\deg F}(1)$ values of $y$: since we are assuming $F(X,Y)\in \mathbb{F}_{p}(Y)[X]$ is square free it follows that $0\neq \text{disc}(F)\in\mathbb{F}_{p}[Y]$. Moreover,
    \[
   \text{disc}(F)(y)\neq 0\Rightarrow  F(X,y) \text{ is square-free}.
    \]
    Thus $F(X,y)$ is a nonconstant square-free polynomial for all but $O_{\deg F}(1)$ values of $y$.
\end{proof}

In our argument we also need to deal with some exponential sums modulo prime squares.
\begin{lem}
Let $F(X,Y)\in\mathbb{F}_{p}[X,Y]$ be a square free polynomial. Then
\[
\left|\sum_{\substack{x\text{ }(p^{2}) \\F(x,y)\neq 0\text{ }(p)}}e_{p^{2}}(sx)\right|=
\begin{cases}
    0               &\text{ if $p\nmid s,$ or $F(X,y)\equiv 0$}\\
    O_{\deg F}(p)               &\text{ if $(s,p^{2})=p$}\\
    p^{2} +O_{\deg F}(p)          &\text{if $p^{2}|s$, and $F(X,y)\not\equiv 0$.}
\end{cases}
\]
\label{lem : exp_{p}2}
\end{lem}
\begin{proof}
   For $p^{2}|s$, the lemma is trivial, so we may assume $p^{2}\nmid s$. Suppose that $p\nmid s$. Since for every $x\in\mathbb{Z}/p^{2}\mathbb{Z}$ we can find $a,b\in\{0,...,p-1\}$ such that $x=a+pb$, it follows that we have
   \[
   \sum_{\substack{x\text{ }(p^{2}) F(x,y)\neq 0}}e_{p^{2}}(sx)=\sum_{\substack{a\in\{0,...,p-1\}\\ F(a,y)\neq 0}}e_{p^{2}}(sa)\sum_{\substack{b\in\{0,...,p-1\}}}e_{p}(sb)=0,
   \]
   since $p\nmid s$. If $(p^{2},s)=p$, then
   \[
   \sum_{\substack{x\text{ }(p^{2}) F(x,y)\neq 0}}e_{p^{2}}(sx)=p\sum_{\substack{a\in\{0,...,p-1\}\\ F(a,y)\neq 0}}e_{p}(s'a)=O_{\deg F}(p),
   \]
   where $s=ps'$ with $(p,s')=1$.
\end{proof}

\subsection{Hooley's Theorem for exponential sums.}
Both the proofs of Proposition $\ref{prop : Anumber}$ and Lemma $\ref{lem : exp}$ rely on a general procedure due to Hooley
\cite{Ho82} (see also (\cite[Theorem A.1]{BB23})), which allows one to estimate a very general family of exponential sums over a finite field, provided that one can control the second moment of an appropriate counting function. 
\begin{thm}[Hooley]
Let $F, G_{1},\dots,G_{k}\in \mathbb{Z}[X_{1},\dots,X_m]$ be polynomials of degree at most $d$ and let
\[
S=\sum_{\substack{x\in\mathbb{F}_{p}^{m}\\G_{1}(\mathbf{x})=\cdots = G_{k}(\mathbf{x})=0}}e_{p}(F(\mathbf{x})).
\]
For each $r\geq 1$ and $\tau\in\mathbb{F}_{p^{r}}$, write
\begin{equation}\label{eq:Nr}
N_{r}(\tau)=\#\left\{\mathbf{x}\in\mathbb{F}_{p^{r}}^{m}: G_{1}(\mathbf{x})=\dots = G_{k}(\mathbf{x})=0, F(\mathbf{x})=\tau\right\}. 
\end{equation}
If there exist $N_{r}\in\mathbb{R}$ such that
\[
\sum_{\tau\in\mathbb{F}_{p^{r}}}|N_{r}(\tau)-N_{r}|^{2}\ll_{d,k,m} p^{\kappa r},
\]
where $\kappa\in\mathbb{Z}$ is independent of $r$, then $S\ll_{d,k,m} p^{\kappa/2}$.
\label{thm : hooley}
\end{thm}

\subsection{Proof of Lemma $\ref{lem : exp}$.}
We are now ready to prove Lemma $\ref{lem : exp}$ which is a direct consequence of:
\begin{lem}
Let $F(X,Y)\in\mathbb{F}_{p}[X,Y]$ be a square-free polynomial such that 
\[
F(X,Y)=\prod_{i=1}^{n}(X+a_{i}Y+b_{i})\text{ over }\overline{\mathbb{F}}_{p}.
\]
Moreover assume that $a_{i}\neq a_{1}$ for some $i\in\{2,...,n\}$. Then for any $s,t\in\mathbb{F}_{p}$, we have
\[
\left|\sum_{x,y\in\mathbb{F}_{p}}\left(\frac{F(x,y)}{p}\right)e_{p}(sx+ty)\right|\ll_{n}p(p,s,t)^{\frac{1}{2}}.
\]
\label{lem : expnotshift}
\end{lem}
\begin{proof}
We start by dealing with the case when $s=t=0$: we have
\[
\sum_{x,y\in\mathbb{F}_{p}}\left(\frac{F(x,y)}{p}\right)=\sum_{y\in\mathbb{F}_{p}}\sum_{x\in\mathbb{F}_{p}}\left(\frac{F(x,y)}{p}\right).
\]
Thanks to Lemma $\ref{lem : exp1dim}$, we get
\begin{equation}
\left|\sum_{x\in\mathbb{F}_{p}}\left(\frac{F(x,y)}{p}\right)\right|\ll_{n}\sqrt{p},
\label{eq : boundcharsum}
\end{equation}
for all but $O_{n}(1)$ values of $y$. Let us denote by $S$ the set of those $y$ for which the bound $(\ref{eq : boundcharsum})$ does not hold. Hence,
\[
\begin{split}
\sum_{y\in\mathbb{F}_{p}}\sum_{x\in\mathbb{F}_{p}}\left(\frac{F(x,y)}{p}\right)&=\sum_{\substack{y\in\mathbb{F}_{p}\setminus S}}\sum_{x\in\mathbb{F}_{p}}\left(\frac{F(x,y)}{p}\right)+\sum_{\substack{y\in S}}\sum_{x\in\mathbb{F}_{p}}\left(\frac{F(x,y)}{p}\right)\\&\ll_{n}  p^{\frac{3}{2}}.
\end{split}
\]
Now we deal with the case $(s,t)\neq (0,0)$. We may assume $t\neq 0$ (the case $s\neq 0$ is similar). We start by observing that
\[
\sum_{x,y\in\mathbb{F}_{p}}\left(\frac{F(x,y)}{p}\right)e_{p}(sx+ty)=\sum_{\substack{w,x,y\in\mathbb{F}_{p}\\w^{2}=F(x,y)}}e_{p}(sx+ty).
\]
With the goal of using Theorem $\ref{thm : hooley}$ , we are going to compute 
\[
\begin{split}
N_{m}(k)&=\#\{(w,x,y)\in\mathbb{F}_{p^{m}}^{3}: w^{2}=F(x,y),\text{ }sx+ty=k\}\\&=\#\{(w,x)\in\mathbb{F}_{p^{m}}^{2}: w^{2}=F(x,s'x+k')\},
\end{split}
\]
for every $m\geq 1$, where $s'=-\overline{t}s$, and $k'=\overline{t}k$. We claim that $F(X,s'X+k')$ is square-free for all but $O_{n}(1)$ values of $k'$. With this in hand, let us see how we can finish the proof of the Lemma. Call $E_{m}$ the set of $k'\in\mathbb{F}_{p^{m}}$ for which the polynomial $F(X,s'X+k')$ is not square-free. For $k'\notin E_m$, the polynomial $F(X,s'X+k')$ is nonconstant and square-free, as verified below. Hence $W^2-F(X,s'X+k')$ is absolutely irreducible, and the Lang--Weil bound gives
\[
\#\{(w,x)\in\mathbb{F}_{p^{m}}^{2}: w^{2}=F(x,s'x+k')\}=p^{m}+O_{n}(p^{\frac{m}{2}}).
\]
Thus,
\[
\begin{split}
\sum_{k\in\mathbb{F}_{p^{m}}}|N_{m}(k)-p^{m}|^{2}&=\sum_{k\in\mathbb{F}_{p^{m}}\setminus E_{m}}|N_{m}(k)-p^{m}|^{2}+\sum_{k\in E_{m}}|N_{m}(k)-p^{m}|^{2}\\&\ll_{n} p^{2m} +\#E_{m}p^{2m}\\&\ll_{n} p^{2m}.
\end{split}
\]
An application of Theorem $\ref{thm : hooley}$ with $\kappa=2$ gives
\[
\left|\sum_{x,y\in\mathbb{F}_{p}}\left(\frac{F(x,y)}{p}\right)e_{p}(sx+ty)\right|\ll_{n}p.
\]
It remains to prove that $F(X,s'X+k')$ is square-free for all but $O_{n}(1)$ values of $k'$. First  observe that
\[ 
F(X,s'X+k')=A_{s',k'}G_{s',k'}(X),
\]
where 
\[
A_{s',k'}=\prod_{i:a_{i}s'=-1}(b_{i}+a_{i}k'),\quad G_{s',k'}(X)=\prod_{\substack{i=1\\a_{i}s'\neq -1}}^{n}((1+a_{i}s')X+a_{i}k'+b_{i}).
\]
Since $A_{s',k'}=0$ only for $O_{n}(1)$ values of $k'$, we may assume that $A_{s',k'}\neq 0$. Moreover, since we have that there exists $i$ such that $a_{i}\neq a_{1}$, then $G_{s',k'}(X)$ is a non-constant polynomial. Moreover, $F(X,s'X+k')$ is square-free if and only if $G_{s',k'}(X)$ is. On the other hand, $G_{s',k'}(X)$ is not square-free if and only if two distinct factors
\[
(1+a_{i}s')X+a_{i}k'+b_{i},\quad (1+a_{j}s')X+a_{j}k'+b_{j}
\]
are proportional. These two linear polynomials are proportional if and only if
\[
\frac{a_{i}k'+b_{i}}{(1+a_{i}s')}=\frac{a_{j}k'+b_{j}}{(1+a_{j}s')}
\]
which is true if and only if
\[
k'(a_{i}-a_{j})=b_{j}(1+a_{i}s')-b_{i}(1+a_{j}s').
\]
If $a_{i}-a_{j}\neq 0$, the equation above is satisfied for a unique value of $k'$. If $a_{i}=a_{j}$, the equation above becomes $(b_{j}-b_{i})(1+a_{i}s')=0$ which would imply $b_{j}-b_{i}=0$ since every factor appearing in the decomposition of $G_{s',k'}(X)$ satisfies $1+a_{i}s'\neq 0$. On the other hand, since we are assuming that $F(X,Y)$ is square-free, if two factors
\[
X+a_{i}Y+b_{i},\quad X+a_{j}Y+b_{j},
\]
are such that $a_{i}=a_{j}$, then necessarily one has $b_{i}-b_{j}\neq 0$. Hence, $G_{s',k'}(X)$ is square free for all but $O_{n}(1)$ values of $k'$ as claimed.

\end{proof}
\begin{proof}[Proof of Lemma $\ref{lem : exp}$.]
    Since $F_{\bfz}=N_{K}(X_{1}+X_{2}\alpha_{2}+z)$, with $z=z_{3}\alpha_{3,K}+\cdots+ z_{n}\alpha_{n,K}$, it follows that
    \[
    F_{\bfz}(X_{1},X_{2})=\prod_{\sigma\in\Emb(K/\mathbb{Q})}(X_{1}+X_{2}\sigma(\alpha_{2,K})+\sigma(z)).
    \]
    Now let $F_{\bfz}=H_{\bfz}^{2}G_{\bfz}$ with $H_{\bfz},G_{\bfz}\in \mathbb{Q}[X,Y]$, monic and $G_{\bfz}$ square-free. Since
    \[
\sum_{x_{1},x_{2}\in\mathbb{F}_{p}}\left(\frac{F_{\bfz}(x_{1},x_{2})}{p}\right)e_{p}(sx_{1}+tx_{2})=\sum_{x_{1},x_{2}\in\mathbb{F}_{p}}\left(\frac{G_{\bfz}(x_{1},x_{2})}{p}\right)e_{p}(sx_{1}+tx_{2})+ O_{n}(p),
    \]
    it is enough to show that $G_{\bfz}(X_{1},X_{2})$ satisfies the hypotheses of Lemma $\ref{lem : expnotshift}$. Since we are assuming $F_{\bfz}\mod p$ is not a square, it follows that
    \[
    G_{\bfz}(X_{1},X_{2})=\prod_{\sigma\in\mathcal{E}}(X_{1}+X_{2}\sigma(\alpha_{2,K})+\sigma(z)),
    \]
    for some $\emptyset\neq \mathcal{E}\subset\Emb(K/\mathbb{Q})$. Let $\sigma\in \mathcal{E}$. Since $\sigma(\alpha_{2,K})\not\in\mathbb{Q}$, we can find $\tau\in\Emb(L/\mathbb{Q})$ such that $\tau(\sigma(\alpha_{2,K}))\neq\sigma(\alpha_{2,K})$. On the other hand,
    \[
    \tau(X_{1}+X_{2}\sigma(\alpha_{2,K})+\sigma(z))|\tau(G_{\bfz})=G_{\bfz}.
    \]
    Hence, $G_{\bfz}$ satisfies the hypotheses of Lemma $\ref{lem : expnotshift}$, since it has two linear factors
    \[
    X_{1}+X_{2}\sigma(\alpha_{2,K})+\sigma(z),\quad X_{1}+X_{2}\tau\circ\sigma(\alpha_{2,K})+\tau\circ\sigma(z),
    \]
    such that $\sigma(\alpha_{2,K})\neq\tau\circ\sigma (\alpha_{2,K})$. Since $p\nmid\disc f_{\alpha_{2,K}}$, the reductions of $\sigma(\alpha_{2,K})$, and $\tau\circ\sigma (\alpha_{2,K})$ are also distinct modulo $p$. Thus $G_{\bfz}\bmod p$ satisfies the hypotheses of Lemma \ref{lem : expnotshift}.
\end{proof}

\subsection{Proof of Proposition $\ref{prop : Anumber}$}
This section is devoted to bounding the exponential sum
\[
\sum_{(x,y,z)\in\mathbb{F}_{p}^{3}}\left(\frac{F(x,y)}{p}\right)\left(\frac{F(z,y+r)}{p}\right)e_{p}(s(x-z)+ty)
\]
for a given square-free polynomial $F(X,Y)$. We begin by establishing the following general estimate:
\begin{lem}
Let $F\in\mathbb{F}_{p}[X,Y]$ be a square-free polynomial with $\deg_{X}F\geq 1$. Then
\[
\left|\sum_{(x,y,z)\in\mathbb{F}_{p}^{3}}\left(\frac{F(x,y)}{p}\right)\left(\frac{F(z,y+r)}{p}\right)e_{p}(s(x-z)+ty)\right|\ll_{\deg F}  p^{2}
\]
\label{lem : exptrivial}
\end{lem}
\begin{proof}
    We start by rewriting
    \[
    \begin{split}
        \sum_{(x,y,z)\in\mathbb{F}_{p}^{3}}&\left(\frac{F(x,y)}{p}\right)\left(\frac{F(z,y+r)}{p}\right)e_{p}(s(x-z)+ty)\\&=\sum_{y\in\mathbb{F}_{p}}e_{p}(ty)\left( \sum_{x\in\mathbb{F}_{p}}\left(\frac{F(x,y)}{p}\right)e_{p}(sx)\right)\left( \sum_{z\in\mathbb{F}_{p}}\left(\frac{F(z,y+r)}{p}\right)e_{p}(-sz)\right).
    \end{split}
    \]
    Thanks to Lemma $\ref{lem : exp1dim}$, there exists $S\subset\mathbb{F}_{p}$, of cardinality $\#S=O_{\deg F}(1)$ and such that
    \[
    \left|\sum_{x\in\mathbb{F}_{p}}\left(\frac{F(x,y)}{p}\right)e_{p}(sx)\right|\ll_{\deg F} \sqrt{p},\quad \left|\sum_{z\in\mathbb{F}_{p}}\left(\frac{F(z,y+r)}{p}\right)e_{p}(-sz)\right|\ll_{\deg F} \sqrt{p},
    \]
    for every $y\not\in S$. Then
    \[
      \begin{split}
        \sum_{(x,y,z)\in\mathbb{F}_{p}^{3}}&\left(\frac{F(x,y)}{p}\right)\left(\frac{F(z,y+r)}{p}\right)e_{p}(s(x-z)+ty)\\&=\sum_{y\in\mathbb{F}_{p}\setminus S}e_{p}(ty)\left( \sum_{x\in\mathbb{F}_{p}}\left(\frac{F(x,y)}{p}\right)e_{p}(sx)\right)\left( \sum_{z\in\mathbb{F}_{p}}\left(\frac{F(z,y+r)}{p}\right)e_{p}(-sz)\right)\\&\quad +\sum_{y\in S}e_{p}(ty)\left( \sum_{x\in\mathbb{F}_{p}}\left(\frac{F(x,y)}{p}\right)e_{p}(sx)\right)\left( \sum_{z\in\mathbb{F}_{p}}\left(\frac{F(z,y+r)}{p}\right)e_{p}(-sz)\right)\\&\ll_{\deg F}p^{2}.
    \end{split}
    \]

\end{proof}

\subsubsection{Notions of the formalism of $\ell$-adic trace functions}
We now recall some notions of the formalism of $\ell$-adic trace functions. We are assuming the definition of a constructible $\ell$-adic sheaf, for which, as well as for a general introduction on the subject, we refer the reader to \cite[Section $1.1$]{Del80} and \cite[Section $3.4.2$]{Kat80}. We start by introducing the notion of a trace function attached to a constructible $\ell$-adic sheaf as in \cite[$7.3.7$]{Kat90}. \newline Let $p,\ell >2$ be distinct primes, $q=p^{k}$ for some $k\geq 1$ and fix an embedding $\iota:\overline{\mathbb{Q}}_{\ell}\hookrightarrow \mathbb{C}$. Let $\mathcal{F}$ be a constructible $\ell$-adic sheaf on $\mathbb{A}_{\mathbb{F}_{q}}^{1}$. For any $x\in\mathbb{A}_{\mathbb{F}_{q}}^{1}(\mathbb{F}_{q^{n}})$ one defines
\[
t_{\mathcal{F},n}(x):=\iota(\Tr(\Fr_{q^{n}}|\mathcal{F}_{\overline{x}})),
\]
where $\Fr_{q^{n}}$ is the geometric Frobenius automorphism of $\mathbb{F}_{q^{n}}$ and $\mathcal{F}_{\overline{x}}$ is the stalk of $\mathcal{F}$ at a geometric point $\overline{x}$ over $x$. The function $t_{\mathcal{F},n}$ is called \textit{the trace function attached to $\mathcal{F}$ over $\mathbb{F}_{q^{n}}$}. If there is no ambiguity, we denote by $t_{\mathcal{F}}$ the trace function $t_{\mathcal{F},1}$. 
\begin{defin}
Let $\mathcal{F}$ be a constructible $\ell$-adic sheaf on $\mathbb{A}_{\mathbb{F}_{q}}^{1}$ and $j:U\hookrightarrow\mathbb{A}_{\mathbb{F}_{q}}^{1}$ the largest dense open subset of $\mathbb{A}_{\mathbb{F}_{q}}^{1}$ where $\mathcal{F}$ is lisse.
\begin{itemize}
\item[$i)$] The sheaf $\mathcal{F}$ is said to be a \textit{middle-extension} $\ell$-adic sheaf if $\mathcal{F}=j_{*}j^{*}\mathcal{F}$ (see \cite[$4.4$, $4.5$]{Kat80} for the definition of $j_{*}j^{*}\mathcal{F}$ and its basic properties).
\item[$ii)$] The sheaf $\mathcal{F}$ is said to be a \textit{middle-extension $\ell$-adic sheaf, punctually pure of weight $0$} if it is a middle-extension sheaf and if for every $n\geq 1$ and every $x\in U(\mathbb{F}_{q^{n}})$, the images of the eigenvalues of $(\Fr_{q^{n}}|\mathcal{F}_{\overline{x}})$ via the fixed embedding $\iota:\overline{\mathbb{Q}}_{\ell}\hookrightarrow \mathbb{C}$, are complex numbers of modulus $1$.
\end{itemize}
\end{defin}

\begin{rem} 
 By \cite[Definition $1.12$]{FKM15}, we will refer to trace functions attached to punctually pure of weight $0$ middle-extension $\ell$-adic sheaves on $\mathbb{A}_{\mathbb{F}_{q}}^{1}$ as \textit{trace functions}.
\end{rem}

\begin{defin}[\cite{FKM19}, pp. $4-6$] Let $\mathcal{F}$ be a constructible $\ell$-adic sheaf on $\mathbb{A}_{\mathbb{F}_{q}}^{1}$ and $j:U\hookrightarrow\mathbb{A}_{\mathbb{F}_{q}}^{1}$ the largest dense open subset of $\mathbb{A}_{\mathbb{F}_{q}}^{1}$ where $\mathcal{F}$ is lisse. The \textit{conductor of $\mathcal{F}$} is defined as
\[
c(\mathcal{F}):=\rank(\mathcal{F})+|\sing(\mathcal{F})|+\sum_{x\in \mathbb{P}_{\mathbb{F}_{q}}^{1}(\overline{\mathbb{F}}_{q})}\text{Swan}_{x}(j_{*}j^{*}\mathcal{F})+\dim H_{c}^{0}(\mathbb{A}_{\overline{\mathbb{F}}_{q}}^1,\mathcal{F}),
\]
where
\begin{enumerate}
\item[$i)$] $\rank(\mathcal{F}):=\dim \mathcal{F}_{x}$, for any $x$ where $\mathcal{F}$ is lisse.
\item[$ii)$] $\sing(\mathcal{F}):=\{x\in\mathbb{P}_{\mathbb{F}_{q}}^{1}(\overline{\mathbb{F}}_{q}):\mathcal{F} \text{ is not lisse at }x\}$.
\item[$iii)$] For any $x\in\mathbb{P}^{1}_{\mathbb{F}_{q}}(\overline{\mathbb{F}}_{q})$, $\Swan_{x}(j_{*}j^{*}\mathcal{F})$ is the \textit{Swan conductor of $\mathcal{F}$ at $x$} (see \cite[Chapter $1$]{Kat88} for the definition of the Swan conductor).
\end{enumerate}
\label{defn : cond}
\end{defin}

\begin{rem}
Recall that if $\mathcal{F}$ is a middle-extension sheaf then
\[
c(\mathcal{F})=\rank(\mathcal{F})+|\sing(\mathcal{F})|+\sum_{x\in {\mathbb{P}}_{\mathbb{F}_{q}}^{1}(\overline{\mathbb{F}}_{q})}\text{Swan}_{x}(\mathcal{F}),
\]
since in this case $\mathcal{F}\cong j_{*}j^{*}\mathcal{F}$ and $\dim H_{c}^{0}(\mathbb{A}_{\overline{\mathbb{F}}_{q}}^{1},\mathcal{F})=0$.
\end{rem}
Finally, we recall the following construction:
\begin{esem}
Let $F(X,Y)\in\mathbb{F}_{q}[X,Y]$ be a square-free polynomial of degree $n$ such that $\deg_YF,\deg_XF\geq 1$. For any nontrivial multiplicative character $\chi$ of $\mathbb{F}_{q}^\times$, consider the function $t:\mathbb{F}_{q}\rightarrow\mathbb{C}$ given by
\[
x \mapsto \sum_{y\in\mathbb{F}_{q}}\chi(F(x,y)).
\]
Combining \cite[Definition $5.10$, Theorem $5.11$]{FKM19} and arguing like in \cite[Lemma $5.1$]{FKM19} one shows that there exists a middle extension $\ell$-adic sheaf $\mathcal{H}_{F,\chi}$ mixed of weight $\leq 1$ such that 
\[
t_{\mathcal{H}_{F,\chi},k}(x)=-\sum_{y\in\mathbb{F}_{q^{k}}}\chi_{k}(F(x,y))\quad\text{ for all } x\in\mathbb{F}_{q^{k}}\setminus( E_{F,\chi}\cap \mathbb{F}_{q^{k}}),
\]
where $\chi_{k}=\chi\circ\text{Norm}_{\mathbb{F}_{q^{k}}/\mathbb{F}_{q}}$. Moreover, one has that
\begin{itemize}
\item[$i)$] $\sing(\mathcal{H}_{F,\chi})\cap\mathbb{A}_{\mathbb{F}_{q}}^{1}(\overline{\mathbb{F}_{q}})\subset E_{F,\chi}$, where
\begin{equation}
E_{F,\chi}=\{x\in\mathbb{A}_{\mathbb{F}_{q}}^{1}(\overline{\mathbb{F}}_{q}): F(x,Y)\text{ is not square-free, or} \deg_{Y} F(x,Y)<\deg_{Y}F\}
\label{eq : eccezionalset}
\end{equation}
\item[$ii)$] $\rank (\mathcal{H}_{F,\chi})=\dim( H_{c}^{1}(\mathbb{A}_{\overline{\mathbb{F}}_{q}}^{1},\mathcal{L}_{\chi(F(x,Y))}))$ for any $x\not\in E_{F,\chi}$.
\item[$iii)$] $c(\mathcal{H}_{F,\chi})\ll_{n} 1$.
\end{itemize}
\label{es : 2vartrace}
\end{esem}

\subsubsection{The key estimate}

\begin{lem}
Let $F(X,Y)\in\mathbb{F}_{q}[X,Y]$ be a square-free polynomial such that
\[
F(X,Y)=\prod_{i=1}^{n}(X+a_{i}Y+b_{i})\text{ over }\overline{\mathbb{F}}_{q}.
\]
Moreover, assume that the following conditions hold:
\begin{itemize}
\item[$i)$] There exists $i\in\{2,...,n\}$ such that $a_{1}\neq a_{i}$.
\item[$ii)$] $\frac{a_{i}-a_{1}}{a_{j}-a_{1}}\cdot(b_{j}-b_{1})\neq b_{i}-b_{1}$ for any $i,j$ such that $a_{j}\neq a_{1}$ and $i\neq j$, $i\neq 1$.
\end{itemize}
Then the sheaf $\mathcal{H}_{F,\chi}$ does not have any irreducible component pure of weight $1$ that is geometrically trivial for every non trivial multiplicative character $\chi$ of $\mathbb{F}_{q}^{\times}$, in particular
\[
\left|\sum_{(x,y)\in\mathbb{F}_{q}^{2}}\chi (F(x,y))\right|\ll_{n} q,
\]
for every non-trivial multiplicative character $\chi$ of $\mathbb{F}_{q}^{\times}$.
\label{lem : multchar}
\end{lem}

\begin{proof}
Let $\mathcal{H}_{F,\chi}^{ss}=\bigoplus_{j}\mathcal{F}_{j}$ be the decomposition of the semisimplification of $\mathcal{H}_{F,\chi}$ into irreducible components. Define $J=\{j: \mathcal{F}_{j}\cong_{\text{geom}}\overline{\mathbb{Q}}_{\ell},\quad \mathcal{F}_{j}\text{ pure of weight }1\}$. For any $r\geq 1$, we have
\[
\sum_{x\in\mathbb{F}_{q^{r}}}t_{\mathcal{H}_{F,\chi},r}(x)= q^{\frac{3r}{2}}\sum_{j\in J}w_{j}^{r}+O(q^{r}),
\]
where $|w_{j}|=1$ for every $j\in J$, thanks to Deligne's resolution of the Riemann hypothesis over finite fields (\cite[Theorem $5.2$]{FKMS19}). On the other hand, for every $m>0$, we have
\[
\begin{split}
\frac{1}{R}\sum_{r\leq R}\left|\frac{1}{q^{\frac{3mr}{2}}}\sum_{x\in\mathbb{F}_{q^{mr}}}t_{\mathcal{H}_{F,\chi},mr}(x)\right|^{2}&=\frac{1}{R}\sum_{r\leq R}\left|\sum_{j\in J}w_{j}^{mr}\right|^{2}+O(R^{-1})
\end{split}
\]
Since $|w_{j}|=1$ for all $j\in J$ we get
\[
\begin{split}
\lim_{R\rightarrow\infty}\frac{1}{R}\sum_{r\leq R}\left|\sum_{j\in J}w_{j}^{mr}\right|^{2}=\#\{(i,j)\in J^{2}: w_{i}^{m}=w_{j}^{m}\}\geq \# J.
\end{split}
\]
which leads to
\begin{equation}
\lim_{R\rightarrow \infty}\frac{1}{R}\sum_{r\leq R}\left|\frac{1}{q^{\frac{3mr}{2}}}\sum_{x\in\mathbb{F}_{q^{mr}}}t_{\mathcal{H}_{F,\chi},mr}(x)\right|^{2}\geq \#J.
\label{eq : m}
\end{equation}
We claim that we can find $m$ such that for every $r\geq 1$,
\[
\sum_{x\in\mathbb{F}_{q^{mr}}}t_{\mathcal{H}_{F,\chi},mr}(x)=O_{n}(q^{mr}).
\]
Assuming the claim, $(\ref{eq : m})$ implies that $\#J=0$, as stated in the Lemma. Hence, it remains to prove the claim. First, observe that, for every $m\geq 1$
\[
t_{\mathcal{H}_{F,\chi},m}(x)=-\sum_{y\in\mathbb{F}_{q^{m}}}\chi_{m}(F(x,y))+O(1)\quad\text{ for all } x\in\mathbb{F}_{q^{m}}\setminus (E_{F,\chi}\cap \mathbb{F}_{q^{m}}),
\]
where 
\[
E_{F,\chi}=\{x\in\mathbb{A}_{\mathbb{F}_{q}}^{1}(\overline{\mathbb{F}}_{q}): F(x,Y)\text{ is not square-free, or} \deg_{Y} F(x,Y)<\deg_{Y}F\}.
\]
So it is enough to prove that for a suitable $m\geq 1$ and for every $r\geq 1$,
\[
\sum_{x,y\in\mathbb{F}_{q^{mr}}}\chi_{mr}(F(x,y))=O_{n}(q^{mr}).
\]
Let $m$ be an integer such that $F(X,Y)$ totally splits over $\mathbb{F}_{q^{m}}$. Then for every $r\geq 1$ we have
\[
F(X,Y)=\prod_{i=1}^{n}(X+a_{i}Y+b_{i}),\text{ over $\mathbb{F}_{q^{mr}}$}.
\]
Hence
\[
\sum_{x,y\in\mathbb{F}_{q^{mr}}}\chi_{mr}(F(x,y))=\sum_{x,y\in\mathbb{F}_{q^{mr}}}\prod_{i=1}^{n}\chi_{mr}(x+a_{i}y+b_{i}).
\]
Using the change of variables $z=x+a_{1}y$, we can rewrite
\[
F(x,y)=g(z)G(z,y),
\]
where
\[
g(z)=\prod_{i:a_{i}=a_{1}}(z+b_{i}),\quad G(z,y)=\prod_{i:a_{i}\neq a_{1}}(z+(a_{i}-a_{1})y+b_{i}).
\]
Since we are assuming that 
\begin{itemize}
    \item[$\bullet$] the polynomial 
    \[
F(X,Y)=\prod_{i=1}^{n}(X+a_{i}Y+b_{i})\text{ over }\overline{\mathbb{F}}_{q},
    \]
    is square free, and hence whenever $a_{i}=a_{j}$ for $i\neq j$, one has $b_{i}\neq b_{j}$,
    \item[$\bullet$] $a_{i}\neq a_{1}$ for some $i\in\{2,...,n\}$,
\end{itemize}
it follows that $G(Z,Y)$ is a non-constant square-free polynomial. Thus we can rewrite the sum above as
\[
\sum_{z,y\in\mathbb{F}_{q^{mr}}}\chi_{mr}(g(z))\chi_{mr}(G(z,y))=-\sum_{z\in\mathbb{F}_{q^{mr}}}t_{\mathcal{L}_{\chi_{mr}(g(Z))}}(z)t_{\mathcal{H}_{G,\chi_{mr}}}(z)+O_{n}(q^{mr}),
\]
where:
\begin{itemize}
    \item[$\bullet$] $\mathcal{L}_{\chi_{mr}(g(Z))}$ is the Kummer sheaf over $\mathbb{A}_{\mathbb{F}_{q^{mr}}}^{1}$ associated with the function $\chi_{mr}(g(z))$, with singularities $\mathrm{Sing}(\mathcal{L}_{\chi_{mr}(g(Z))})\cap\mathbb{A}_{\mathbb{F}_{q}}^{1}(\overline{\mathbb{F}_{q}})=\{-b_{i}:i\text{ such that }a_{i}=a_{1}\}$.
    \item[$\bullet$] the sheaf $\mathcal{H}_{G,\chi_{mr}}$ is the one in Example $\ref{es : 2vartrace}$ associated to the polynomial $G(Z,Y)$, i.e.
\[\sing(\mathcal{H}_{G,\chi_{mr}})\cap\mathbb{A}_{\mathbb{F}_{q}}^{1}(\overline{\mathbb{F}}_{q})\subset E_{G,\chi}=\{z\in\mathbb{A}_{\mathbb{F}_{q}}^{1}(\overline{\mathbb{F}}_{q}): G(z,Y)\text{ is not square-free}\}.
\]
\end{itemize}
We claim that the sheaf $\mathcal{H}_{G,\chi_{mr}}$ is lisse at $-b_{1}$. To do this, it is enough to show that the polynomial $G(-b_{1},Y)$ is square-free: for distinct $i,j$ such that $a_{j}\neq a_{1}$, the factors
\[
(a_{i}-a_{1})Y+(b_{i}-b_{1}),\quad (a_{j}-a_{1})Y+(b_{j}-b_{1})
\]
are proportional only if
\[
\frac{a_{i}-a_{1}}{a_{j}-a_{1}}\cdot(b_{j}-b_{1})= b_{i}-b_{1},
\]
but this is not possible thanks to hypothesis $(ii)$. Since $-b_{1}\in\mathrm{Sing}(\mathcal{L}_{\chi_{mr}(g(Z))})$ but $-b_{1}\not\in \sing(\mathcal{H}_{G,\chi_{mr}})$ and $\mathcal{L}_{\chi_{mr}(g(Z))}$ is irreducible, no sheaf in the decomposition $\mathcal{H}_{G,\chi_{mr}}^{ss}=\bigoplus_{j}\mathcal{F}_{j}$ is geometrically isomorphic to $\mathcal{L}_{\chi_{mr}(g(Z))}^{\vee}$. Then, the quasi-orthogonality of trace functions (\cite[Theorem $5.2$]{FKMS19}) implies that
\[
\sum_{x,y\in\mathbb{F}_{q^{mr}}}\chi_{mr}(F(x,y))=-\sum_{z\in\mathbb{F}_{q^{mr}}}t_{\mathcal{L}_{\chi_{mr}(g(Z))}}(z)t_{\mathcal{H}_{G,\chi_{mr}}}(z)+O_{n}(q^{mr})\ll_{n} q^{mr},
\]
as claimed.
\end{proof}

We are finally ready to prove:

\begin{proposition*}
Let $p$ be an odd prime and let $F(X,Y)\in\mathbb{F}_p[X,Y]$ be a square-free polynomial such that
\[
F(X,Y)=\prod_{i=1}^{n}(X+a_{i}Y+b_{i})\text{ over }\overline{\mathbb{F}}_{p},
\]
and assume that the following conditions hold:
\begin{itemize}
\item[$i)$] There exists $i\in\{2,...,n\}$ such that $a_{1}\neq a_{i}$.
\item[$ii)$] $\frac{a_{i}-a_{1}}{a_{j}-a_{1}}\cdot(b_{j}-b_{1})\neq b_{i}-b_{1}$ for any $i,j$ such that $a_{j}\neq a_{1}$ and $i\neq j$, $i\neq 1$.
\item[$iii)$] For every $i,j$ with $i\neq j$, $b_{i}\neq b_{j}$.
\item[$iv)$] For every $y\in\overline{\mathbb{F}}_{p}$, the polynomial $F(X,y)$ is not a square over $\overline{\mathbb{F}}_{p}[X]$.
\end{itemize}

 Let $s\in\mathbb{F}_{p}$, then there exists $A_{s}\subset\{(r,t)\in\mathbb{F}_{p}\times\mathbb{F}_{p}\}$ such that
\begin{itemize}
\item[$i)$] for any $(r,t)\in (\mathbb{F}_{p}\times\mathbb{F}_{p})\setminus A_{s}$, one has 
\[
\left|\sum_{(x,y,z)\in\mathbb{F}_{p}^{3}}\left(\frac{F(x,y)}{p}\right)\left(\frac{F(z,y+r)}{p}\right)e_{p}(s(x-z)+ty)\right|\ll_{n} p^{3/2}(s,t,p)^{1/2}.
\]
\item[$ii)$] $\# A_{s}=O_{n}(1)$.
\end{itemize}
\end{proposition*}
\begin{proof}
The case $s=t=0$ is a direct consequence of Lemma \ref{lem : exptrivial}. We first treat the case $s\neq 0$. We begin by noticing that:
\[
\sum_{x,y,z\in\mathbb{F}_{p}^{3}}\left(\frac{F(x,y)}{p}\right)\left(\frac{F(z,y+r)}{p}\right)e_{p}(s(x-z)+ty)=\sum_{\substack{w,x,y,z\in\mathbb{F}_{p}^{4}\\w^{2}=F(x,y)F(z,y+r)}}e_{p}(s(x-z)+ty).
\]
Then, thanks to Theorem $\ref{thm : hooley}$, it is enough to estimate
\[
\begin{split}
N_{m}(k)&=|\{(w,x,y,z)\in\mathbb{F}_{p^{m}}^{4}: w^{2}=F(x,y)F(z,y+r),\text{ } s(x-z)+ty=k \}|\\
&=|\{(w,x,y)\in\mathbb{F}_{p^{m}}^{3}: w^{2}=F(x,y)F(x+\overline{s}(ty-k),y+r)\}|,
\end{split}
\]
for any $m\geq 1$ and any $k\in\mathbb{F}_{p^{m}}$.
First observe that over $\overline{\mathbb{F}}_{p}$ we can decompose $F(X,Y)F(X+\overline{s}(tY-k),Y+r)$ as
\begin{equation}
\begin{split}
F(X,Y)F(X+\overline{s}(tY-k),Y+r)& =\prod_{i=1}^{n}(X+a_{i}Y+b_{i})\prod_{i=1}^{n}\left(X+(a_{i}+\overline{s}t)Y+b_{i}+a_{i}r-\overline{s}k\right).
\end{split}
\label{eq : pol}
\end{equation}
 We first construct a set
$A_s\subset\overline{\mathbb{F}}_{p}^{\,2}$ with $\# A_s=O_{n}(1)$ such that, whenever $(r,t)\in\mathbb{F}_{p}^2\setminus A_s$, the polynomial $F(X,Y)F(X+\overline{s}(tY-k),Y+r)$ is not a square over $\overline{\mathbb{F}}_{p}$ for any $k\in\overline{\mathbb{F}}_{p}$. The polynomial $F(X,Y)$ is square-free by hypothesis. The polynomial $F(X+\overline{s}(tY-k),Y+r)$ is also square-free since it is obtained by composing $F(X,Y)$ with a suitable linear change of variables.  From the decomposition in $(\ref{eq : pol})$ it follows that if $F(X,Y)F(X+\overline{s}(tY-k),Y+r)$ is a square, then there exists a permutation $g\in S_{n}$ such that for every $i=1,...,n$ the factors
 \[
 X+a_{i}Y+b_{i},\quad X+(a_{g(i)}+\overline{s}t)Y+b_{g(i)}+a_{g(i)}r-\overline{s}k,
 \]
 are proportional. If $g=\Id$, then clearly $t=r=k=0$. If $g\in S_{n}\setminus\{\Id\}$, we deduce that for every $i=1,...,n$ one has
 \[
 \begin{cases}
     s(a_{i}-a_{g(i)})-t=0\\
     s(b_{g(i)}-b_{i}+a_{g(i)}r)=k.
 \end{cases}
 \]
Now let $h\in\{1,...,n\}$ such that $a_{1}\neq a_{h}$. Consider $i, j$ such that $1=g(i)$ and $h=g(j)$. Then the following equations hold
\[
 \begin{cases}
     s(a_{i}-a_{1})-t=0\\
     s(b_{1}-b_{i}+a_{1}r)=k\\
     s(a_{j}-a_{h})-t=0\\
     s(b_{h}-b_{j}+a_{h}r)=k.
 \end{cases}
 \]
 \[
 \begin{cases}
 t=s(a_{i}-a_{1}),\\
 r=\frac{b_{1}+b_{j}-b_{i}-b_{h}}{a_{h}-a_{1}}
 \end{cases}
 \]
Now, if we define
\[
A_{s}=\left\{\left(\frac{b_{1}+b_j-b_i-b_h}{a_h-a_{1}}, s(a_i-a_{1})\right):1\leq i,j\leq n\right\}\cup\{(0,0)\},
\]
it follows that for any $(r,t)\not\in A_{s}$ and any $k\in\mathbb{F}_{p^{m}}$ the surface defined by the polynomial $W^{2}= F(X,Y)F(X+\overline{s}(tY-k),Y+r)$ is absolutely irreducible. Hence, using the Lang-Weil bound, we get that
\begin{equation}
N_{m}(k)=p^{2m}+O_{n}(p^{3m/2}),
\label{eq : lang_weil}
\end{equation}
for any $k\in\mathbb{F}_{p^{m}}$, provided that $(r,t)\not\in A_{s}$. From now on, we always assume that $(r,t)\not\in A_{s}$. Moreover, to lighten the notation, we denote by $s'=\overline{s}t$, $k'=\overline{s}k$.
The next step is to show that the polynomial $F(X,Y)F(X+s'Y-k',Y+r)$ satisfies the hypotheses of Lemma \ref{lem : multchar} for all but $O_{n}(1)$ choices of $k$. Notice that the polynomials $F(X,Y)$ and $F(X+s'Y-k',Y+r)$ both satisfy the hypotheses of Lemma $\ref{lem : multchar}$: the polynomial $F(X,Y)$ by hypothesis, and $F(X+s'Y-k',Y+r)$ because it is obtained from $F(X,Y)$ using a change of variables. \newline
We start by showing that the polynomial in $(\ref{eq : pol})$ is square-free for all but $O_{n}(1)$ many $k'$: since both $F(X,Y)$ and $F(X+s'Y-k',Y+r)$ are square-free, we need only check that for all but $O_{n}(1)$ values of $k'$ and for every $i,j$, the polynomials
\begin{equation}
 X+a_{i}Y+b_{i},\quad X+(a_{j}+s')Y+b_{j}+a_{j}r-k',
\label{eq : condsq}
\end{equation}
are not proportional, which is guaranteed if $k'\neq b_{j}-b_{i}+a_{j}r$.\newline 
Next, we verify that the polynomial $F(X,Y)F(X+s'Y-k',Y+r)$ also satisfies the second condition in Lemma \ref{lem : multchar} for all but $O_{n}(1)$ many $k'$. Since $F(X,Y)$ and $F(X+s'Y-k',Y+r)$ already satisfy this condition individually by hypothesis, it remains to ensure that any pair of linear factors coming from different components of the product still satisfy condition $(ii)$ of Lemma \ref{lem : multchar}. We start by checking that given two factors
\[
 X+a_{i}Y+b_{i},\quad X+(a_{j}+s')Y+b_{j}+a_{j}r-k',
\]
with $a_{i}\neq a_{1}$ and $a_{j}+s'\neq a_{1}$, they satisfy
\begin{equation}
\frac{b_{j}+a_{j}r-k'-b_{1}}{a_{j}+s'-a_{1}}\neq\frac{b_{i}-b_{1}}{a_{i}-a_{1}},
\label{eq : condition(ii)}
\end{equation}
for all but $O_{n}(1)$ values of $k'$: since we are assuming $a_{1}\neq a_{i}$, the inequality in \eqref{eq : condition(ii)} holds for all but at most one value of $k'$. Since $F(X,Y)$ satisfies the hypotheses of Lemma \ref{lem : multchar}, two factors of it satisfy condition $(ii)$. It remains to check that, given two distinct factors
\[
 X+(a_{i}+s')Y+b_{i}+a_{i}r-k',\quad X+(a_{j}+s')Y+b_{j}+a_{j}r-k',
\]
with $a_{i}+s'\neq a_{1}$, $a_{j}+s'\neq a_{1}$, they satisfy
\begin{equation}
\frac{b_{i}+a_{i}r-k'-b_{1}}{a_{i}+s'-a_{1}}\neq\frac{b_{j}+a_{j}r-k'-b_{1}}{a_{j}+s'-a_{1}},
\label{eq : condition(ii).1}
\end{equation}
for all but $O_{n}(1)$ values of $k'$. If $a_i\neq a_j$, the inequality \eqref{eq : condition(ii).1} is satisfied for all but at most one value of $k'$. If $a_i=a_j$, it reduces to $b_i\neq b_j$, which follows from square-freeness since $i\neq j$.
If a factor other than $X+a_{1}Y+b_{1}$ has slope $a_{1}$, condition $(ii)$ merely requires its constant term to differ from $b_{1}$; this is already guaranteed by the square-freeness exclusions above.
Therefore, for all but $O_{n}(1)$ many $k'$, the polynomial $F(X,Y)F(X+s'Y-k',Y+r)$ satisfies both conditions of Lemma \ref{lem : multchar}.

Hence, if we denote by $\chi_{m}$ the unique character of order $2$ over $\mathbb{F}_{p^{m}}$, it follows from Lemma $\ref{lem : multchar}$ that
\begin{equation}
N_{m}(k)=p^{2m}+\sum_{x,y}\chi_{m}(F(x,y)F(x+s'y-k',y+r))=p^{2m}+O_{n}(p^{m}),
\label{eq : equblabla}
\end{equation}
for all but $O_{n}(1)$ values of $k'$. Let $E_m$ denote the set of $k'$ for which \eqref{eq : equblabla} does not hold. Take
\[
\begin{split}
\sum_{k\in\mathbb{F}_{p^{m}}}|N_{m}(k)-p^{2m}|^{2}& =\sum_{k\in\mathbb{F}_{p^{m}}\setminus E_{m}}|N_{m}(k)-p^{2m}|^{2}+\sum_{k\in E_{m}}|N_{m}(k)-p^{2m}|^{2}\\&\ll_{n} p^{3m}+\#E_{m}p^{3m}\\&\ll_{n} p^{3m},
\end{split}
\]
where in the last step we used \eqref{eq : lang_weil} to bound the contribution of $k\in E_m$. Applying Theorem \ref{thm : hooley} with $\kappa=3$, the desired bound follows for $s\neq 0$ for every $(r,t)\not\in A_{s}\cap (\mathbb{F}_{p})^{2}$. 
\newline We now deal with the case $s=0$, $t\neq 0$. In this case, the exponential sum becomes
\[
\sum_{x,y,z\in\mathbb{F}_{p}^{3}}\left(\frac{F(x,y)}{p}\right)\left(\frac{F(z,y+r)}{p}\right)e_{p}(ty)=\sum_{\substack{w,x,y,z\in\mathbb{F}_{p}^{4}\\w^{2}=F(x,y)F(z,y+r)}}e_{p}(ty).
\]
With the goal of using Hooley's Theorem, for any $m\geq 1$ and any $k\in\mathbb{F}_{p^{m}}$ we have to count
\[
\begin{split}
N_{m}(k)&=|\{(w,x,y,z)\in\mathbb{F}_{p^{m}}^{4}: w^{2}=F(x,y)F(z,y+r),\text{ } ty=k \}|\\
&=|\{(w,x,z)\in\mathbb{F}_{p^{m}}^{3}: w^{2}=F(x,\overline{t}k)F(z,\overline{t}k+r)\}|\\&
=p^{2m}+\sum_{(x,z)\in\mathbb{F}_{p^{m}}^{2}}\chi_{m}(F(x,\overline{t}k)F(z,\overline{t}k+r))\\&= p^{2m}+\left(\sum_{x\in\mathbb{F}_{p^{m}}}\chi_{m}(F(x,\overline{t}k))\right) \left(\sum_{z\in\mathbb{F}_{p^{m}}}\chi_{m}(F(z,\overline{t}k+r))\right).
\end{split}
\]
On the other hand, thanks to hypothesis $(iv)$ we know that for every $k$ the polynomials $F(X,\overline{t}k)$ and $F(Z,\overline{t}k+r)$ are not squares, hence an application of (\cite[Theorem $11.23$]{IK04}) leads to
\[
\left|\sum_{x\in\mathbb{F}_{p^{m}}}\chi_{m}(F(x,\overline{t}k))\right|\ll_{n} p^{\frac{m}{2}},\quad \left|\sum_{z\in\mathbb{F}_{p^{m}}}\chi_{m}(F(z,\overline{t}k+r))\right|\ll_{n} p^{\frac{m}{2}}.
\]
Hence,
\[
N_{m}(k)=p^{2m}+O_{n}(p^{m}),
\]
for every $k$. Hence, applying Theorem $\ref{thm : hooley}$, we get
\[
\left|\sum_{x,y,z\in\mathbb{F}_{p}^{3}}\left(\frac{F(x,y)}{p}\right)\left(\frac{F(z,y+r)}{p}\right)e_{p}(ty)\right|\ll_{n} p^{\frac{3}{2}}.
\]
Thus, when $s=0$, we may take $A_0=\varnothing$.

\end{proof}

\section*{Acknowledgments}

This project began while the author was affiliated with the University of Basel, and it originated as a joint project with Pierre Le Boudec. Although P. Le Boudec has decided to move away from academia and preferred not to be listed as a co-author, the author is deeply grateful for his contribution to the paper. The author would also like to thank Julian Demeio for helpful discussions, which led to a simplification of Lemma $\ref{lem : subspace}$. Thanks are due as well to Tim Browning and Lillian Pierce for their valuable comments on an earlier version of the paper, which helped improve the exposition of the material. Finally, the author is grateful to the anonymous referees for their helpful comments and suggestions.
\section*{AI statement}
The mathematical ideas and proof strategy are due to the author. ChatGPT (GPT-6 Astra) was used to assist with proofreading and checking the mathematical arguments.

\bibliographystyle{alpha}

\end{document}